\documentclass[12pt]{amsart}

\usepackage{amssymb}

\usepackage{enumitem}

\usepackage{graphicx}

\makeatletter
\@namedef{subjclassname@2020}{%
  \textup{2020} Mathematics Subject Classification}
\makeatother

\usepackage[T1]{fontenc}

\usepackage[colorlinks,
linkcolor=blue,
anchorcolor=blue,
citecolor=blue]{hyperref}

\newtheorem{theorem}{Theorem}
\newtheorem{lemma}{Lemma}[section]
\newtheorem{proposition}{Proposition}[section]
\newtheorem{corollary}{Corollary}[section]
\newtheorem{question}{Question}

\theoremstyle{definition}
\newtheorem{definition}{Definition}[section]
\newtheorem{example}{Example}[section]

\newtheorem{remark}{Remark}[section]

\numberwithin{equation}{section}

\begin{document}


\baselineskip=17pt


\title{Which hyponormal block Toeplitz operators are either normal or analytic?}
\author{Senhua Zhu}
\address{Department of Mathematics,
Shantou University,
Shantou, Guangdong, 515063, P. R. China}
\email{shzhu@stu.edu.cn}

\author{Yufeng Lu}
\address{Department of Mathematics Sciences, Dalian University of Technology,
Dalian, Liaoning, 116024, P. R. China}
\email{lyfdlut@dlut.edu.cn}

\author{Chao Zu*}
\address{Department of Mathematics Sciences, Dalian University of Technology,
Dalian, Liaoning, 116024, P. R. China}
\email{zuchao@mail.dlut.edu.cn}

\date{}


\begin{abstract}
In this paper, we continue Curto-Hwang-Lee's work to study the connection between hyponormality and subnormality for block Toeplitz operators acting on the vector-valued Hardy space of the unit circle. Curto-Hwang-Lee's work focuses primarily on hyponormality and subnormality of block Toeplitz operators with rational symbols. By studying the greatest common divisor of matrix-valued inner functions and the ``weak" commutativity of matrix-valued inner functions, we extended Curto-Hwang-Lee's result to block Toeplitz operators with symbols of bounded type.  More precisely, we proved that if $\Psi,\Psi^{\ast}$ are matrix-valued functions of bounded type and the inner part of $\Psi$ of Douglas-Shapiro-Shields factorization is a scalar inner function, then every hyponormal Toeplitz operator $T_{\Psi}$ whose square is also hyponormal must be either normal or analytic.

\end{abstract}

\thanks{*Corresponding author.}

\maketitle

\section{Introduction}

Studying the structure and properties of operators from different perspectives is an important way to recognize and understand operators. This will lead to many new mathematical ideas and methods and generate numerous new mathematical problems.
One approach to understand general operators is to study some ``relatively simple"  special classes of operators.
The study of special classes of operators will provide a great supply of examples for the study of general operators and generate concepts and techniques for the general picture.  For example, the study of compressed shift operators on the Hardy space over the unit disk has given rise to several celebrated and influential theories, such as Nagy-Foias theory on operator models, Ando's dilation theorem of commuting operator pairs. Moreover, the study of compressed shift operators on the Hardy space over polydiscs remains a vital component of multivariable operator theory \cite{LYZ2023,Yang2019,ZYL2020}.

Operators like shift operators, compressed shift operators, Toeplitz operators, normal operators, subnormal operators, hyponormal operators, etc. are important classes of special operators.
Toeplitz operators, block Toeplitz operators, and block Toeplitz determinants arise naturally in several fields of mathematics and various problems in physics (in particular, in the field of quantum mechanics). For example, the spectral theory of Toeplitz operators plays an important role in the study of solvable models in quantum mechanics \cite{Prun2003} and in the study of the one-dimensional Heisenberg Hamiltonian of ferromagnetism \cite{Damak2006};  the theory of block Toeplitz determinants is used in the
study of the classical dimer model \cite{Basor2007} and in the study of the vicious walker model \cite{Hikami2003}; the theory of block Toeplitz operators is also used in the study of Gelfand-Dickey Hierarchies \cite{Cafa2008}. On the other hand, the theory of hyponormal and subnormal operators is a broad and highly developed area, which has made important contributions to a number of problems in functional analysis, operator theory, and mathematical physics (see, for example, \cite{Haya2000, Ifan1971, Szafra2000} for applications to related mathematical physics problems). Thus, it becomes of central significance to describe in detail hyponormality and subnormality for (block) Toeplitz operators \cite{Cowen1988,Cowen,Curto1990,CL2001,CL2002,Gu2002}. This paper is concerned with the gap between hyponormality and subnormality for block Toeplitz operators. More precisely, we address a reformulation of Halmos's Problem 5 in the case of block Toeplitz operators: Which hyponormal block Toeplitz operators are either normal or analytic?

Let $H^{2}$ denote the Hardy space on the unit circle. For $\varphi\in L^{\infty}$ of the unit circle, the Toeplitz operator $T_{\varphi}$ is defined on $H^{2}$ by $T_{\varphi}g=P(\varphi g)$, where $P$ is the orthogonal projection of $L^{2}$ onto $H^{2}$.
Halmos's Problem 5 (see \cite{Halmos1970}) asked:
\begin{center}
{\it Is every subnormal Toeplitz operator either normal or analytic?}
\end{center}
Although the answer was shown to be negative in general by serval authors (Sun \cite{Sun1984,Sun1985}, Cowen and Longe \cite{CL1984}),
an interesting partial affirmative answer to Halmos's Problem 5 was given by Abrahamse \cite{Abrahamse1976} under some suitable reformulation.
 The symbol $\varphi$ is said to be of $bounded~type$ (or in the Nevanlinna class) if there are analytic functions $\psi_{1},\psi_{2}\in H^{\infty}$ such that $\varphi=\psi_{1}/\psi_{2}$ almost everywhere on the unit circle. Functions of bounded type have been studied extensively in the literatures (e.g. \cite{Abrahamse1976,CL2001,DSS1970}) and these studies have made a significant contribution to establishing the rich interactions between function theory and operator theory.
 \newtheorem*{abra}{Abrahamse's Theorem}
\begin{abra}
 Suppose $\varphi \in L^{\infty}$ and $\varphi$ or $\overline{\varphi}$ is of bounded type.
If
\begin{description}
  \item[i] $T_{\varphi}$ is hyponormal,
  \item[ii] $\ker[T_{\varphi}^{\ast},T_{\varphi}]$ is invariant for $T_{\varphi}$,
\end{description}
 then $T_{\varphi}$ is either normal or analytic.
 \end{abra}
In 1988, the hyponormality of the Toeplitz operator $T_{\varphi}$ is characterized by a property of the symbol $\varphi$ via Cowen's Theorem \cite{Cowen1988}.
 \newtheorem*{acow}{Cowen's Theorem}
\begin{acow}
Suppose $\varphi \in L^{\infty}$,
$T_{\varphi}$ is hyponormal if and only if there exists $K\in H^{\infty}$ with $\|K\|\leq 1$ such that $\varphi+\overline{\varphi}K \in H^{\infty}$.

\end{acow}

This paper focuses primarily on hyponormality and subnormality of block Toeplitz operators with symbols of bounded type. For the general theory of subnormal and hyponormal operators, we refer to \cite{conway1991,mp1989}.
To describe our results, we first need to review a few essential facts and some notations about block Toeplitz operators, and it can be found in \cite{CHL2021,NF,Nikolskii1986,Niko,Peller2003}.

 Let $B(D,E)$ denote the set of bounded linear operators from the Hilbert space $D$ to $E$, $B(E):=B(E,E)$. If $\dim(E)=n$, then $B(E)$ is the set of $n\times n$ matrices denoted by $\mathbb{M}_{n}$. Suppose $A,B \in B(E)$ and define $[A,B]:=AB-BA$. $A$ is said to be \textit{normal} if $[A^{\ast}, A]=0$, \textit{hyponormal} if $[A^{\ast}, A]\geq 0$, and \textit{subnormal} if $A$ has a normal extension, i.e., $A=N\mid_{E}$, where $N$ is a normal operator on some Hilbert space $H \supset E  $ such that $E$ is an invariant subspace of $N$.

Suppose $E$ is a complex separable Hilbert space, let $L^{2}(E)$ be the $E$-valued norm square-integrable measurable functions on $\mathbb{T}$, and let $H^{2}(E)$ be the corresponding Hardy space. In particular, $L^{2}(\mathbb{C}^{n})=L^{2} \otimes \mathbb{C}^{n}$, $H^{2}(\mathbb{C}^{n})=H^{2} \otimes \mathbb{C}^{n}$.
Suppose $\Phi$ is a matrix-valued function in $ L^{\infty}(\mathbb{M}_{n})=L^{\infty}\otimes \mathbb{M}_{n}$, the \textit{block Toeplitz operator} $T_{\Phi}$ with the symbol $\Phi$ is defined by
\begin{equation*}
  \begin{split}
    T_{\Phi}: H^{2}(\mathbb{C}^{n})& \rightarrow H^{2}(\mathbb{C}^{n})\\
     f&\mapsto T_{\Phi}f=P(\Phi f),
  \end{split}
\end{equation*}
where $P$ denotes the orthogonal projection from $L^{2}(\mathbb{C}^{n})$ onto $H^{2}(\mathbb{C}^{n})$. Let $J$ be the unitary operator on $L^{2}(\mathbb{C}^{n})$ defined by
\begin{equation*}
  \begin{split}
    J: L^{2}(\mathbb{C}^{n}) &\rightarrow L^{2}(\mathbb{C}^{n})\\
     f&\mapsto Jf=\overline{z}f(\overline{z}),
  \end{split}
\end{equation*}
then $J^{\ast}=J$ and $J^{2}=I$.
The \textit{block Hankel operator} $H_{\Phi}$ with the symbol $\Phi$ is defined by
\begin{equation*}
  \begin{split}
    H_{\Phi}: H^{2}(\mathbb{C}^{n})& \rightarrow H^{2}(\mathbb{C}^{n})\\
     f &\mapsto H_{\Phi}f=JP_{-}(\Phi f),
  \end{split}
\end{equation*}
where $P_{-}=I-P$. If we set $ H^{2}(\mathbb{C}^{n})=H^{2}\oplus \ldots \oplus H^{2} $ then we see that
\begin{equation}\nonumber
 T_{\Phi}=\left(
  \begin{array}{ccc}
    T_{\varphi_{11}} & \cdots & T_{\varphi_{1n}} \\
    \vdots & \vdots & \vdots \\
    T_{\varphi_{n1}} & \cdots & T_{\varphi_{nn}}
  \end{array}
\right),
H_{\Phi}=\left(
  \begin{array}{ccc}
    H_{\varphi_{11}} & \cdots & H_{\varphi_{1n}} \\
    \vdots & \vdots & \vdots \\
    H_{\varphi_{n1}} & \cdots & H_{\varphi_{nn}}
  \end{array}
\right),
\end{equation}
where
\begin{equation}\nonumber
 \Phi=\left(
  \begin{array}{ccc}
    \varphi_{11} & \cdots & \varphi_{1n} \\
    \vdots & \vdots & \vdots \\
    \varphi_{n1} & \cdots & \varphi_{nn}
  \end{array}
\right) \in L^{\infty}(\mathbb{M}_{n}).
\end{equation}
For a matrix-valued function $\Phi=(\varphi_{ij}) \in L^{\infty}(\mathbb{M}_{n})$, $\Phi$ is said to be of \textit{bounded type} if all entries $\varphi_{ij},1\leq i,j\leq n$ of $\Phi$ are of bounded type, and we say that $\Phi$ is \textit{rational} if each $\varphi_{ij}$ is a rational function.

In 2006, Gu, Hendricks and Rutherford \cite{GHR2006}  extended Cowen's theorem to the matrix-valued setting. They characterized the hyponormality of block Toeplitz operators in terms of their symbols. Their characterization for hyponormality of block Toeplitz operator resembles Cowen's Theorem except for an additional condition---the normality condition of the symbol.
 \newtheorem*{agu}{Hyponormality of block Toeplitz operators}
\begin{agu}
Let $\Phi \in L^{\infty}(\mathbb{M}_{n})$. The block Toeplitz operator $T_{\Phi}$ is
 hyponormal if and only if the following two conditions hold:
 \begin{description}
   \item[i] $\Phi$ is normal, i.e. $\Phi^{\ast}(z)\Phi(z)=\Phi(z)\Phi^{\ast}(z)$ for almost every $z\in \mathbb{T}$.
   \item[ii] There exists a matrix $K(z) \in H^{\infty}(\mathbb{M}_{n})$ such that $\|K\|_{\infty} \leq 1$ and $\Phi-K\Phi^{\ast} \in H^{\infty}(\mathbb{M}_{n})$.
 \end{description}
 there exists $K\in H^{\infty}$ with $\|K\|\leq 1$ such that $\Phi+\overline{\Phi}K \in H^{\infty}$.

\end{agu}

A matrix-valued function $\Theta \in H^{\infty}(\mathbb{M}_{n\times m})(=H^{\infty} \otimes \mathbb{M}_{n\times m})$ is said to be $inner$ if $\Theta(z)$ is isometric almost everywhere on $\mathbb{T}$, i.e., $\Theta^{\ast}\Theta=I_{m}$ a.e. on $\mathbb{T}$. In particular, if $n=m$, then $\Theta$ is said to be \textit{two-sided inner}, i.e., $\Theta(z)$ is unitary almost everywhere on $\mathbb{T}$.
The concept of inner function plays an important role in characterized invariant subspaces of shifted operator on the vector-valued Hardy space.

\newtheorem*{blh}{Beurling-Lax-Halmos Theorem}
\begin{blh}[matrix-valued version]
A nonzero closed subspace $M$ of $H^{2}(\mathbb{C}^{n})$ is invariant for the shifted operator $S$ on $H^{2}(\mathbb{C}^{n})$ if and only if
$$M=\Theta H^{2}(\mathbb{C}^{m}),$$
where $\Theta$ is an inner matrix-valued function in  $H^{\infty}(\mathbb{M}_{n\times m})(m\leq n)$. Furthermore, $\Theta$ is unique up to a unitary constant right factor, i.e., if $M=\Theta'H^{2}(\mathbb{C}^{r})$ for an inner function $\Theta'$ in $H^{\infty}(\mathbb{M}_{n\times r})(r\leq n)$, then $m=r$ and there exists a constant unitary operator $U$ from $\mathbb{C}^{m}$ onto $\mathbb{C}^{m}$ such that $\Theta=\Theta'U$. (As customarily done, we will say that two matrix-valued functions $A$ and $B$ are \textit{equal} if they are equal up to a unitary constant right factor.)
\end{blh}
 Beurling-Lax-Halmos Theorem  certainly plays a central role in establishing important early interplay results such as characterizing kernels of Hankel operators in terms of the bounded type-ness of the associated matrix-valued symbols. For the scalar setting, it was known (Lemma3 in \cite{Abrahamse1976}) that if $\varphi \in L^\infty$, $\varphi$ is of bounded type if and only if $\ker H_\varphi \neq \{0\}$. For the matrix-valued setting, it was shown by Gu, Hendricks and Rutherford (Theorem 2.2 in \cite{GHR2006}) that for $\Phi\in L^\infty(\mathbb{M}_{n})$, $\Phi$ is of bounded type if and only if $\ker H_\varphi = \Theta H^{2}(\mathbb{C}^{n})$ for some two-side inner function $\Theta$.

On the other hand, the Bram-Halmos criterion for subnormality \cite{Bram1955, conway1991} states that an operator $T\in B(\mathcal{H})$
is subnormal if and only if $\sum_{i,j} \langle T^i x_j, T^j x_i\rangle \geq 0$ for all finite collections $x_0,x_1,\cdots,x_k\in \mathcal{H}$. It
is easy to see that this is equivalent to the following positivity test:

\begin{equation}\label{sh1.1}
\left(
  \begin{array}{cccc}
    [T^{*},T] & [T^{*2},T] & \cdots & [T^{*k},T] \\
    ~[T^{*},T^2] &[T^{*2},T^2]  & \cdots & [T^{*k},T^2] \\
    \vdots & \vdots & \ddots & \vdots \\
     ~[T^{*},T^k]& [T^{*2},T^k] & \cdots & [T^{*k},T^k]\\
  \end{array}
\right)
 \geq 0 ~~~(all~~k\geq 1).
\end{equation}
The positivity condition \eqref{sh1.1} for $k=1$ is equivalent to the hyponormality of $T$, while
subnormality requires the validity of \eqref{sh1.1} for all $k\in \mathbb{N}$.  For $k \geq 1$, an operator $T$ is said to be
\textit{$k$-hyponormal} if $T$ satisfies the positivity condition \eqref{sh1.1} for a fixed $k$ \cite{Curto1990}. Thus the
Bram-Halmos criterion can be stated as: $T$ is subnormal if and only if $T$ is $k$-hyponormal for all $k\geq 1$. The notion of $k$-hyponormality has been considered by many authors with an aim at understanding the gap between hyponormality and subnormality \cite{JL2001, MP1992, CL2005, CLY2005,CL2003}. The Bram-Halmos criterion indicates
that hyponormality is generally far from subnormality. For example, in \cite{Halmos1982} (see Problem 209), it was shown that there exists a hyponormal operator whose square is not hyponormal, e.g., $S^*+2S$ ($S$ is the canonical unilateral shift), which is a trigonometric Toeplitz operator, i.e., $S^*+2S=T_{\bar{z}+2z}$ . This example addresses the gap between hyponormality and subnormality for Toeplitz operators. But there are special classes of operators for which $k$-hyponormality and subnormality are equivalent.  In \cite{CL2001} (see Example 3.1), it was shown
that there is no gap between $2$-hyponormality and subnormality for back-step extensions of
recursively generated subnormal weighted shifts. Also in \cite{CL2001}, as a partial answer, it was shown by Curto and Lee that every hyponormal Toeplitz operator with trigonometric polynomial symbol  whose square is hyponormal must be either normal or
analytic. Furthermore, in \cite{CL2003}, it was shown that this result still holds for Toeplitz operators  with  bounded type symbol.

Recently,  Curto, Hwang and Lee, et al. \cite{CHKL2014,CKL2012,CHL2012,CHL2019} are devoted to  investigating the hyponormality and the subnormality of Toeplitz operators with matrix-valued symbols. They have tried to answer a number of questions involving matrix
functions of bounded type and have made a great progress on rational functions.
However, there are still many questions that are difficult to answer. The main problem lies on the cases where $\theta$ is a singular inner function.  For example, they embarked on the gap between hyponormality and subnormality, and studied an appropriate reformulation of Halmos's Problem 5: which hyponormal block Toeplitz operators are either normal or analytic? One of their results listed as follow:
\newtheorem*{achl}{Curto-Hwang-Lee's Theorem}
\begin{achl}[Theorem 4.5 in \cite{CHL2012}]
Let $\Phi\in L^{\infty}(\mathbb{M}_{n})$ be a matrix-valued rational function. Define $\Phi_{-}:=[(I-P)\Phi]^{\ast}$, then we may write
$$\Phi_{-}=B^{\ast}\Theta, $$
where $B\in H^{2}(\mathbb{M}_{n})$ and $\Theta=\theta I_{n}$ with a finite Blaschke product $\theta$. Suppose $ B $ and $\Theta$ are coprime. If both $T_{\Phi}$ and $T_{\Phi}^{2}$  are hyponormal then $T_{\Phi}$ is either normal or analytic.
\end{achl}
 Their approach involved utilizing the zeros of the finite Blaschke product. However, this approach is not applicable to general scalar inner functions because the singular inner has no zero in $\mathbb{D}$.
Hence they were unable to decide whether this result still holds for matrix-valued bounded type symbols.

\newtheorem*{achll}{Curto-Hwang-Lee's Problem}
\begin{achll}[Problem 10.6 in \cite{CHL2019}]
Let $\Phi \in L^{\infty}(\mathbb{M}_{n})$ be such that $\Phi$ and $\Phi^{\ast}$ are of bounded type of the form $$ \Phi_{-}=\theta B^{\ast}~~(\text{coprime}),$$
where $\theta$ is an inner function in $H^{\infty}$ and $\Phi_{-}=[(I-P)\Phi]^{\ast}$. If $T_{\varphi}$ and $T_{\varphi}^{2}$ are hyponrmal, does it follow that $T_{\Phi}$ is either normal or analytic?
\end{achll}
In this paper, we are interesting in Curto-Hwang-Lee's Problem and we observe that the key point to solve this problem is to figure out the greatest common divisor of matrix-valued inner functions.
For the scalar setting, the greatest common divisor of Blaschke products is a Blaschke product generated by their common zeros, and the greatest common divisor of singular inner functions is determined by their associated singular measures.  For the matrix-valued setting, the left inner divisor of finite Blaschke products is recently studied by Curto, Hwang and Lee in \cite{CHL2022}. In section 3, we mainly investigate the greatest common divisor of matrix-valued inner functions and obtain the following result:


\newtheorem*{ashth}{Theorem \ref{shly10}}
\begin{ashth}
Suppose $\Delta_{i}, i=1,2$ are two-sided inner functions with values in  $\mathbb{M}_{n}$ without nonconstant scalar inner factors and $D(\Delta_{i})=\delta, i=1,2$, where $\delta$ is a scalar inner function. If $\theta$ is a scalar inner function and
\begin{equation}\nonumber
  \text{left-g.c.d}\{\theta I_{n}, \Delta_{1} \Delta_{2}\}=\text{left-g.c.d}\{\theta I_{n}, \delta\Delta_{1}  \}=\Omega,
\end{equation}
then $\Omega$ has the form $\omega I_{n}$ for some scalar-valued inner function $\omega$. Moreover,
$$\omega =\text{g.c.d.}\{\theta, \delta\},$$ and $\omega^{\ast} \theta$ and $\delta $ are coprime.
\end{ashth}

If $\theta$ is a finite Blaschke product, by taking suitable $\Delta_{1},\Delta_{2}$ associated with the symbol of $T_{\Phi}$,  Curto-Hwang-Lee proved that $\omega=1$ (see \cite{CHL2012}), and then obtained that Problem 10.6 is true. However, if $\theta$ is not a finite Blaschke product, Curto-Hwang-Lee's method is not directly effective.  But now, thanks to Theorem \ref{shly10}, we can deal with the cases where $\theta$ is a singular inner function and solve the Problem 10.6.
\newtheorem*{shth}{Theorem \ref{shth2}}
\begin{shth}
Let $\Phi \in L^{\infty}(\mathbb{M}_{n})$ be such that $\Phi$ and $\Phi^{\ast}$ are of bounded type of the form
\begin{equation}\nonumber
  \Phi=\Phi_{1} \theta^{\ast}, ~(\theta I_n,\Phi_{1}~\text{are right coprime}),
\end{equation}
where $\theta$ is a scalar inner function and $\Phi_{1} \in H^{\infty}(\mathbb{M}_{n})$. If $T_{\Phi}$ and $T^{2}_{\Phi}$ are hyponormal, then $T_{\Phi}$ is either normal or analytic.
 \end{shth}



The remainder of the paper is organized as follows: In section 2, we review the notions and the definitions related to the matrix-valued functions and  give some essential lemmas about minimal scalar inner multiple $D(\Theta)$.
 In section 3, we first study the connection between the minimal scalar inner multiple and the $\text{left-g.c.d.}$ and $\text{left-l.c.m.}$ of two-sided inner functions, and then give the proof of Theorem \ref{shly10}, and finally we discuss the ``weak" commutativity of two-sided inner functions.
 In section 4, by using Theorem \ref{shly10} and the ``weak" commutativity, we give an affirmative answer to Problem 10.6 in \cite{CHL2019} and the main result is Theorem \ref{shth2}.

\section{Preliminaries}

\subsection{Factorization of matrix-valued inner functions}
Let $\mathcal{P}_{n}$ be the set of all polynomials with values in $\mathbb{C}^{n}$. A matrix-valued function $f\in H^{2}(\mathbb{M}_{n})$ is said to be $outer$ if $f \mathcal{P}_{n}$ is dense in $H^{2}(\mathbb{C}^{n})$. The analogue of the scalar Inner-Outer factorization theorem on the classical Hardy space can be stated as follows:
\newtheorem*{iof}{\bf Inner-Outer Factorization for $H^{2}(\mathbb{M}_{n})$-functions}
\begin{iof}
If $f \in H^{2}(\mathbb{M}_{n})$, then $f$ can be expressed in the form
$$f=f_{i}f_{o},$$
where $f_{o}$ is an outer function with values in $\mathbb{M}_{n}$ and $f_{i}$ is an inner function with values in $\mathbb{M}_{n}$.
\end{iof}

It is worth noting that the factorization of  matrix-valued inner functions is more complicated compared to the scalar setting due to noncommutativity in matrix multiplication.

 Let $\{\theta_{i},i\in I\}$ be a family of scalar inner functions. The \textit{ greatest common inner divisor} $\theta_{d}$ and the $least ~common ~inner ~multiple$ $\theta_{m}$ of the family of $\{\theta_{i},i\in I\}$ are the inner functions defined by
$$\theta_{d}H^{2}=\bigvee_{i \in I}\theta_{i}H^{2}~\text{and}~\theta_{m}H^{2}=\bigcap_{i\in I}\theta_{i}H^{2},$$
where $\bigvee$ denotes the closed linear span on $H^{2}$.
Similarly, let $\{\Theta_{i},i\in I\}$ be a family of inner functions with values in $\mathbb{M}_{n}$. The the $greatest~ common~ left~ inner~ divisor$ $\Theta_{ld}$ and the $least~ common~ left~ inner~ multiple$ $\Theta_{lm}$ of the family $\{\Theta_{i},i\in I\}$ are the inner functions defined by
$$\Theta_{ld}H^{2}(\mathbb{C}^{n})=\bigvee_{i \in I}\Theta_{i}H^{2}(\mathbb{C}^{n})~~\text{and}~~\Theta_{lm}H^{2}(\mathbb{C}^{n})=\bigcap_{i\in I}\Theta_{i}H^{2}(\mathbb{C}^{n}).$$
The \textit{greatest common right inner divisor} $\Theta_{rd}$ of $\{\Theta_{i},i\in I\}$ defined as the greatest common left inner divisor of $\{\widetilde{\Theta_{i}},i\in I\}$ and the \textit{least common right inner multiple} $\Theta_{rm}$ of $\{\Theta_{i},i\in I\}$ defined as the least common left inner divisor of $\{\widetilde{\Theta_{i}},i\in I\}$.
The Beurling-Lax-Halmos Theorem guarantees that $\Theta_{ld}$ and $\Theta_{lm}$ are unique up to a unitary constant right factor, and $\Theta_{rd}$ and $\Theta_{rm}$ are unique up to a unitary constant left factor.

\begin{remark}\label{shre5}
For the scalar setting, the least common inner divisor always exists and is nonzero. For example, if $u_{1}$ and $u_{2}$ are scalar inner functions, then we have
$$\{0\} \neq u_{1}u_{2}H^{2}\subset u_{1}H^{2} \cap u_{2}H^{2}.  $$
 However, for the matrix-valued setting, this property does not hold (see \cite{CHL2021,ZLY}), there exist two inner functions $\Theta_{i},i=1,2$ with values in $\mathbb{M}_{m_{1},n}$ and $\mathbb{M}_{m_{2},n}$ such that
$$\Theta_{1}H^{2}(\mathbb{C}^{m_{1}}) \perp \Theta_{2}H^{2}(\mathbb{C}^{m_{2}}).$$
In other words, the invariant subspaces $\Theta_{1}H^{2}(\mathbb{C}^{m_{1}})$ and $ \Theta_{2}H^{2}(\mathbb{C}^{m_{2}})$ are orthogonal to each other.
\end{remark}

For the scalar setting,
let $\text{g.c.d.}\{\cdot\}$ and $\text{l.c.m.}\{\cdot \}$ denote the greatest common inner divisor and the least common inner multiple, respectively. For the matrix-valued setting, let $\text{left-g.c.d.}\{\cdot\}(\text{right-g.c.d.}\{\cdot\})$ and $\text{left-l.c.m.}\{\cdot\}$ $(\text{right-l.c.m.}\{\cdot\})$ denote the greatest common left(right) inner divisor and the least common left(right) inner multiple, respectively.
Two matrix-valued inner functions
$\Theta_{1}$ and $\Theta_{2}$ are said to be \textit{left coprime} if $\text{left-g.c.d.}\{\Theta_{1},\Theta_{2}\}=1$. $\Theta_{1}$ and $\Theta_{2}$ are said to be \textit{right coprime} if $\widetilde{\Theta_{1}}$ and $\widetilde{\Theta_{2}}$ are left coprime.

For a matrix-valued function $\Phi \in H^{2}(\mathbb{M}_{n \times m})$, we say that $\Theta \in H^{2}(\mathbb{M}_{n\times r})$ is a $left~ inner ~divisor$ of $\Phi$ if $\Theta$ is a matrix-valued inner function such that $\Phi=\Theta \Psi$ for some $\Psi \in H^{2}(\mathbb{M}_{r\times m})(r \leq n)$.
Two matrix-valued functions $\Phi$ and $\Psi$  in $ H^{2}(\mathbb{M}_{n})$ are said to be \textit{left coprime} if their inner parts of Inner-Outer factorization $\Phi_{i}$ and $\Psi_{i}$ are left coprime. There are some simple but important facts about the left inner divisors of two-sided inner functions:

\begin{lemma}[Lemma 2.1 in \cite{CHL2012}]
Let $\Theta_{i}=\theta_{i}I_{n}$ for an scalar inner function $\theta_{i}, i\in J$. Let $\theta_{d}=\text{g.c.d.}\{\theta_{i}:i \in J\}$ and $\theta_{m}=\text{l.c.m.}\{\theta_{i}:i \in J\}$, then
\begin{description}
  \item[(a)] $\text{left-g.c.d.}\{\Theta_{i}:i \in J\}=\text{right-g.c.d.}\{\Theta_{i}:i \in J\}=\theta_{d}I_{n}$;
  \item[(b)] $\text{left-l.c.m.}\{\Theta_{i}:i \in J\}=\text{right-l.c.m.}\{\Theta_{i}:i \in J\}=\theta_{m}I_{n}$.
\end{description}
\end{lemma}

\begin{lemma}[Lemma 2.2 in \cite{CHL2021}]\label{shly33}
If $\Theta$ is a matrix-valued two-sided inner function, then any left inner divisor of $\Theta$ is also two-sided inner.
\end{lemma}

Let $\Theta$ be a matrix-valued inner function and $\theta$ be a scalar inner function, for convenience, we will say that $\theta$ is a \textit{scalar inner factor} of $\Theta$ if $\theta I_{n}$ is a left inner divisor of $\Theta$, and that $\Theta$ has \textit{no nonconstant scalar inner factors} if $\Theta$ has no left inner divisor of the form $\theta I_{n}$. Define
$$G(\Theta)=\text{l.c.m.}\{\theta:\theta~\text{is a scalar inner factor of}~\Theta\}.$$
We call $G(\Theta)$ the \textit{greatest scalar inner factor} of $\Theta$. It is clear that $G(\Theta)=1$ if and only if $\Theta$ has no nonconstant scalar inner factors.

\subsection{Matrix-valued functions of bounded type}
For $\Phi \in L^{\infty}(\mathbb{M}_{n})$, define $\widetilde{\Phi}:=\Phi(\overline{z})^{\ast}$. The following relations can be easily derived:
\begin{align}
   \label{eq1} &T^{\ast}_{\Phi}=T_{\Phi^{\ast}},~~H^{\ast}_{\Phi}=H_{\widetilde{\Phi}}.\\
        \label{eq2} &T_{\Phi \Psi}-T_{\Phi}T_{\Psi}=H^{\ast}_{\Phi^{\ast}}H_{\Psi},~\Psi \in L^{\infty}(\mathbb{M}_{n}).\\
        \label{eq3} &H_{\Phi}T_{\Psi}=H_{\Phi \Psi}, ~\Psi \in H^{\infty}(\mathbb{M}_{n}).\\
        \label{eq4} &T^{\ast}_{\widetilde{\Psi}}H_{\Phi}= H_{\Psi \Phi}, ~\Psi \in H^{\infty}(\mathbb{M}_{n}).
\end{align}

Let $\Psi=zI_{n}$, observe by $(\ref{eq3})$ and $(\ref{eq4})$, we have $H_{\Phi}S=H_{\Phi}T_{zI_{n}}=T^{\ast}_{zI_{n}} H_{\Phi}$. It implies that the kernel of a block Hankel operator $H_{\Phi}$ is an invariant subspace of shifts. Thus by Beurling-Lax-Halmos Theorem, if $\ker H_{\Phi}\neq \{0\}$, then $$\ker H_{\Phi} =\Theta H^{2}(\mathbb{C}^{m})$$
for some matrix-valued inner function $\Theta$.

By Theorem 2.2 in \cite{GHR2006} (also Corollary 2.17 in \cite{CHL2021}), if $\Phi \in L^{\infty}(\mathbb{M}_{n}) $, then
 $$ \Phi ~\text{is of bounded type } ~\Leftrightarrow ~ \ker H_{\Phi}=\Theta H^{2}(\mathbb{C}^{n})~ \text{for a two-sided inner function}~\Theta.$$
For the scalar setting, Abrahamse (Lemma 6 in \cite{Abrahamse1976}) showed that if $T_{\Phi}$ is hyponormal, then $\Phi$ is of bounded type if and only if $\overline{\Phi}$ is of bounded type. However, for the matrix-valued setting, the bounded type-ness of $\Phi, \Phi^{\ast}$ is not equivalent even though $T_{\Phi}$ is hyponormal. But we have one-way implication: If $T_{\Phi}$ is hyponormal and $\Phi^{\ast}$ is of bounded type, then $\Phi$ is of bounded type (Corollary 3.5 and Remark 3.6 in \cite{GHR2006}). Thus
whenever we deal with hyponormal Toeplitz operators $T_{\Phi}$ with symbols $\Phi$ satisfying that both
$\Phi$ and $\Phi^{\ast}$ are of bounded type, it suffices to assume
that only $\Phi^{\ast}$ is of bounded type. In spite of this, for convenience, we will assume that $\Phi$ and $\Phi^{\ast}$
are of bounded type whenever we deal with bounded type symbols.

To understand matrix functions of bounded type, we need to factor those functions into a coprime product of matrix inner functions and the adjoints of matrix-valued $H^{\infty}$-functions; this is the so-called Douglas-Shapiro-Shields factorization (\cite{DSS1970,Fuh1981}). This
factorization is very helpful and somewhat unavoidable for the study of Hankel
and Toeplitz operators with bounded type symbols.
\newtheorem*{dss}{\bf Douglas-Shapiro-Shields factorization}
\begin{dss}
A function $\Phi \in L^{\infty}(\mathbb{M}_{n})$ is of bounded type, the there exists a two-sided inner function $\Theta$ and $\Psi \in H^{\infty}(\mathbb{M}_{n})$ such that
\begin{equation}\label{shf78}
  \Phi= \Psi\Theta^{\ast},
\end{equation}
where $ \Theta $ and $\Psi$ are right coprime. We often say that $(\ref{shf78})$ is the \textit{Douglas-Shapiro-Shields factorization} of $\Phi$. $\Theta$ is called the $inner~part$ of this factorization.
\end{dss}

\subsection{The minimal scalar inner multiple $D(\Theta)$}
We introduce a notion of minimal scalar inner multiple (Definition 4.11 in \cite{CHL2019}).

\begin{definition}\label{sdef1}
Suppose $\Theta$ is a two-sided inner function with values in $\mathbb{M}_{n}$. We say that $\Theta$ has a \textit{scalar inner multiple} $\theta$ if there exists $\Delta \in H^{\infty}(\mathbb{M}_{n})$ such that
   $$\Theta \Delta=\Delta\Theta =\theta I_{n}.$$
   Define
$$D(\Theta)=\text{g.c.d.}\{\theta : \theta ~ \text{is a scalar inner multiple of}~ \Theta \}.$$
We call $D(\Theta)$ the \textit{minimal scalar inner multiple} of $\Theta$. In particular, $D(\Theta)$ is a scalar inner multiple of $\Theta$, hence there exists two-sided inner functions $\Theta'$ such that $\Theta \Theta'=D(\Theta) I_{n}$.
\end{definition}
\newtheorem*{arem}{\it Remark}
\begin{arem}
In above definition, it is enough to assume $\Theta \Delta=\theta I_{n}$. For, given two matrices $A$ and $B$ such that $AB = \lambda I$ with $\lambda \neq 0$, it is straightforward to verify
that $BA = AB$.
It is also well known (Proposition 5.6.1 in \cite{NF}) that if $\Theta \in H^{\infty}(\mathbb{M}_{n})$ is inner then $\Theta$ has a scalar inner multiple. Thus $D(\Theta)$ is well-defined.
We would like to
remark that the notion of $D(\Theta)$ arises in the Sz.-Nagy-Foia\c{s} theory of contraction
operators $T$ of class $C_0$ (completely nonunitary contractions $T$ for which there exists
a nonzero function $u \in H^{\infty}$ such that $u(T ) = 0$): the minimal function $m_T$ of the
$C_0$-contraction operator $T$ amounts to $D(\Theta_{T})$, where $\Theta_{T}$ is the characteristic
function of $T$.
\end{arem}
The followings are some useful results about $D(\Theta)$ which can be found in \cite{CHL2012,CHL2019,CHL2021}.
\begin{lemma}[Theorem 4.13 in \cite{CHL2019}, Lemma C.13 in \cite{CHL2021}]\label{shly1}
Let $\Delta \in H^{\infty}(\mathbb{M}_{n})$ be a two-sided inner function and $\theta$ be a scalar inner function. Then the following are equivalent:
\begin{description}
  \item[(a)] $\theta I_{n}$ and $\Delta$ are left coprime;
  \item[(b)] $\theta I_{n}$ and $\Delta$ are right coprime;
  \item[(c)] $\theta $ and $D(\Delta)$ are  coprime;
  \item[(d)] $\theta$ and $\det \Delta$ are  coprime.
\end{description}

\end{lemma}

\begin{lemma}\label{shly2}
Suppose  $\Theta $ and $ \Delta$ are two-sided inner functions with values in $\mathbb{M}_{n}$, then the following statements hold:
\begin{description}
  \item[(a)] $D(\Theta)$ and $D(\Delta)$ are the factors of $D(\Theta \Delta)$;
  \item[(b)] $D(\Theta \Delta)$ is the factor of $D(\Theta)D(\Delta)$;
  \item[(c)] if $D(\Theta)$ and $D(\Delta)$ are coprime, then $D(\Theta \Delta)=D(\Theta)D(\Delta)$.
\end{description}

\end{lemma}
\begin{proof}
By definition \ref{sdef1}, there exists a two-sided inner function $\Omega$ such that
$$D(\Theta \Delta)I_{n}=\Theta \Delta \Omega= \Delta \Omega\Theta,$$
which implies that statement $(a)$ holds.

Similarly,
there exist two-sided inner functions $\Theta_{1}, \Delta_{1}$ such that
$$D(\Theta)I_{n}=\Theta \Theta_{1},~~D(\Delta)I_{n}=\Delta \Delta_{1}.$$
Therefore,
$$\Theta \Delta \Delta_{1} \Theta_{1}=D(\Theta)D(\Delta)I_{n},$$
which yields that (b) holds.

For (c), using statement (a) and (b), there exist scalar inner functions $\theta, \delta$ which are divisors of $D(\Theta)$ and $D(\Delta)$, respectively, such that
$$D(\Theta \Delta)=D(\Theta) \delta=D(\Delta) \theta.$$
Since $D(\Theta)$ and $D(\Delta)$ are coprime, then $D(\Theta )=\theta$ and $D(\Delta)=\delta$. 
\end{proof}

\begin{lemma}\label{shly7}
Suppose $\Theta$ is a two-sided inner functions with values in $\mathbb{M}_{n}$. Write $D(\Theta) I_{n}=\Theta \Theta'$, then
$\Theta'$ has no nonconstant scalar inner factors, i.e. $G(\Theta)=1$.
\end{lemma}
\begin{proof}
Suppose $\Theta'=\delta \Delta_{1}$, where $\delta$ is a nonconstant scalar inner function, then
$$D(\Theta)I_{n}=\Theta \Theta'= \delta \Theta \Delta_{1} \Rightarrow \delta^{\ast}D(\Theta)I_{n}=\Theta \Delta_{1}.$$
This implies that $\delta^{\ast}D(\Theta)$ is a scalar inner multiple of $\Theta$, which contradicts to the minimality of $D(\Theta)$.
\end{proof}
%

\begin{lemma}\label{shly9}
Suppose $\Theta$ and $ \Delta$ are two-sided inner functions with values in  $\mathbb{M}_{n}$ and $\Theta\Delta=\theta I_{n}$ for some scalar inner function $\theta$. Then the following statements hold:
\begin{enumerate}
  \item $D(\Theta)=D(\Delta)$ if and only if $G(\Theta)=G(\Delta)$;
  \item if $\Theta$~(or $\Delta$)  has no nonconstant scalar inner factors, i.e. $G(\Theta)=1$~(or $G(\Delta)=1$),  then
  $$\Theta\Delta=D(\Delta)I_{n}~~(~\text{or } ~~\Theta\Delta= D(\Theta)I_{n}).$$
\end{enumerate}
\end{lemma}
\begin{proof}

It is clear that $D(\Theta)$ and $D(\Delta)$ are the factors of $\Theta\Delta$. Hence there exist scalar inner functions $\omega_{1}$ and $\omega_{2}$ such that
\begin{equation}\label{shf12}
  \Theta \Delta=D(\Theta)\omega_{1}I_{n}=D(\Delta) \omega_{2}I_{n}.
\end{equation}

Write
\begin{equation}\nonumber
  D(\Theta)I_{n} =\Theta \Theta',~~D(\Delta)I_{n}=\Delta \Delta',
\end{equation}
 by $(\ref{shf12})$, we have
 \begin{equation}\label{shf79}
   \Theta=\Delta' \omega_{2},~~\Delta=\Theta' \omega_{1}.
 \end{equation}
According to Lemma \ref{shly7}, $\Theta'$ and $\Delta'$ have no nonconstant scalar inner factors, i.e., $G(\Theta')=G(\Delta')=1$.

For the necessity of (1), since $D(\Theta)=D(\Delta)$, $(\ref{shf12})$ implies that $\omega_{1}=\omega_{2}$. Thus
$$G(\Theta)=G(\Delta' \omega_{2} )=\omega_{2}=\omega_{1}=G( \Theta' \omega_{1})=G(\Delta).$$

For the sufficiency of (1), assume $G(\Theta)=G(\Theta)=\omega$, then there exist $\Theta_{1}$ and $\Delta_{1}$ such that
$$\Theta=\omega\Theta_{1},~\Delta=\omega \Delta_{1},$$
where $G(\Theta_{1})= G(\Delta_{1})=1$. By $(\ref{shf79})$,
$$\omega\Delta_{1}=\omega_{1}\Theta',~\omega\Theta_{1}=\omega_{2}\Delta'.$$
Since $\Theta', \Delta', \Theta_{1}$ and $\Delta_{1}$ have no nonconstant scalar factors, then $$\omega=\omega_{1}=\omega_{2}.$$
By $(\ref{shf12})$, we can see that $D(\Theta)=D(\Delta)$.

For (2),
by $(\ref{shf79})$, we have
$$\Delta=\omega_{1} \Theta',~ \Theta= \omega_{2} \Delta'.$$
Since $\Theta$  has no nonconstant scalar inner divisors, then $\omega_{2}=1$. By $(\ref{shf12})$ again, $D(\Delta)I_{n}=\Theta\Delta$. The similar argument holds for $\Delta$.
\end{proof}

\begin{lemma}\label{shly27}
Suppose $\Theta, \Delta$ are two-sided inner functions with values in  $\mathbb{M}_{n}$ without nonconstant scalar inner factors. Let $D(\Theta)I_{n}=\Theta \Theta_{1}$ and $D(\Delta)I_{n}=\Delta \Delta_{1}$, then $$D(\Theta \Delta)=D(\Theta) D(\Delta) $$ if and only if $\Delta_{1} \Theta_{1}$ have no nonconstant scalar inner factors.
\end{lemma}
\begin{proof}
Write $D(\Theta \Delta)I_{n} =\Theta \Delta \Omega$ and $D(\Theta) D(\Delta)=D(\Theta \Delta) \omega$, then we have
$$\Theta \Delta \Omega \omega=D(\Theta) D(\Delta) I_{n}= \Theta D(\Delta)\Theta_{1} =\Theta \Delta \Delta_{1} \Theta_{1} .$$
This implies that $\Omega \omega= \Delta_{1} \Theta_{1}$. By Lemma \ref{shly7}, $G(\Omega)=1$, thus $ \omega=1$ if and only if $G(\Delta_{1} \Theta_{1})=1$.

%
\end{proof}

\section{Two-sided inner functions}

\subsection{The least common inner multiple of two-sided inner functions}
As we have mentioned in Remark \ref{shre5}, for general matrix-valued inner functions, the least common inner multiple will be zero, the following is a explicit example.
\begin{example}
Suppose
\begin{equation}\nonumber
  \Theta=\left(
  \begin{array}{c}
    z^{3} \\
    0 \\
  \end{array}
\right), ~\Delta=\left(
  \begin{array}{c}
    0\\
    z^{3} \\
  \end{array}
\right)
\end{equation}
are inner functions with values in  $\mathbb{M}_{2\times 1}$. It is not hard to check that
$$\Theta H^{2}(\mathbb{C}) \perp \Delta H^{2}(\mathbb{C}).$$
Hence $\text{left-l.c.m.}\{\Theta, \Delta\}=0$.
\end{example}

The following lemma shows that the behavior of two-sided inner functions is different, that is, $\text{left-l.c.m.}\{\Theta, \Delta\}\neq 0$ when $\Theta, \Delta$ are two-sided inner.
\begin{lemma}\label{shly32}
Suppose $\Theta$ and $\Delta$ are two-sided inner functions with values in  $\mathbb{M}_{n}$, and let $$\Omega=\text{left-l.c.m.}\{\Theta, \Delta\}.$$
Then the following statements hold:
\begin{description}
  \item[(a)] $\Omega$ is a nonzero two-sided inner function;
  \item[(b)] there exist two-sided inner functions $\Theta_{1}$ and $\Delta_{1}$ that are right coprime such that
  \begin{equation}\label{shf72}
    \Omega=\Theta \Theta_{1}=\Delta \Delta_{1}.
  \end{equation}
\end{description}
\end{lemma}
\begin{proof}
For (a):
Let $D(\Theta) I_{n}=\Theta \Theta'$, $D(\Delta)I_{n}=\Delta \Delta'$ and
$$\omega=\text{l.c.m.}\{D(\Theta), D(\Delta)\},$$
then there exist scalar inner functions $\omega_{1}, \omega_{2}$ such that
\begin{equation}\label{shf55}
 \Theta \Theta' \omega_{1}= D(\Theta)\omega_{1} I_{n}=\omega I_{n}=D(\Delta)\omega_{2} I_{n}=\Delta \Delta' \omega_{2}.
\end{equation}
It implies that
$$  \Theta H^{2}(\mathbb{C}^{n}) \supset \Theta \Theta' \omega_{1} H^{2}(\mathbb{C}^{n})=\omega H^{2}(\mathbb{C}^{n}) =\Delta \Delta' \omega_{2} H^{2}(\mathbb{C}^{n}) \subset \Delta H^{2}(\mathbb{C}^{n}). $$
Therefore,
$$ \omega H^{2}(\mathbb{C}^{n})\subset \Theta H^{2}(\mathbb{C}^{n}) \cap \Delta H^{2}(\mathbb{C}^{n})= \Omega H^{2}(\mathbb{C}^{n}) , $$
which yields that $\Omega$ is a left inner divisor of $\omega$. Hence $\Omega \neq 0$. By Lemma \ref{shly33}, $\Omega$ is a two-sided inner function.

For (b): Since
 $$ \Omega H^{2}(\mathbb{C}^{n})=\Theta H^{2}(\mathbb{C}^{n}) \cap \Delta H^{2}(\mathbb{C}^{n}),$$
then there exist two-sided inner functions $\Theta_{1},\Delta_{1}$ such that $(\ref{shf72})$ holds. Let
$$\Lambda =\text{right-g.c.d.}\{\Theta_{1}, \Delta_{1}\},$$
then $\Theta_{1}\Lambda^{\ast}$ and $\Delta_{1} \Lambda^{\ast}$ are two-sided inner functions.
By $(\ref{shf72})$, $\Theta \Theta_{1} \Lambda^{\ast} =\Delta \Delta_{1} \Lambda^{\ast}=\Omega \Lambda^{\ast}$, which means $\Theta$ and $\Delta$ are both left inner divisors of $\Omega \Lambda^{\ast}$, and so $\Omega$ is a left inner divisor of $\Omega \Lambda^{\ast}$. Therefore, $\Lambda$ is a constant unitary matrix, which implies that $\Theta_{1}$ and $\Delta_{1}$ are right coprime.
\end{proof}

\begin{lemma}\label{shly23}
Suppose $\Theta$ and $\Delta$ are two-sided inner functions with values in  $\mathbb{M}_{n}$ and $\Theta\Delta= \Delta \Theta$,
then $$\text{left-l.c.m.}\{\Theta, \Delta\} =\Theta \Delta \Omega^{\ast},$$
where $\Omega=\text{right-g.c.d.}\{\Theta, \Delta\}$.
\end{lemma}
\begin{proof}
Let $\Lambda= \text{left-l.c.m.}\{\Theta, \Delta\}$, then there exist two-sided inner functions $\Lambda_{1}$ and $ \Lambda_{2}$  such that
\begin{equation}\label{shf56}
  \Lambda= \Theta \Lambda_{1}=\Delta \Lambda_{2}.
\end{equation}
Since $\Theta \Delta \Omega^{\ast}=\Delta\Theta \Omega^{\ast}$, then
$$\Theta \Delta \Omega^{\ast} H^{2}(\mathbb{C}^{n}) \subset  \Theta H^{2}(\mathbb{C}^{n}) \cap \Delta H^{2}(\mathbb{C}^{n})=\Lambda H^{2}(\mathbb{C}^{n}) ,$$
which means that there exists a two-sided inner function $\Lambda_{3}$ such that
\begin{equation}\label{shf57}
\Theta \Delta \Omega^{\ast}=\Lambda \Lambda_{3}=\Delta \Theta \Omega^{\ast}.
\end{equation}
By $(\ref{shf56})$ and $(\ref{shf57})$ we obtain
$$\Theta \Omega^{\ast}=\Lambda_{2}\Lambda_{3}~\text{and}~\Delta \Omega^{\ast} =\Lambda_{1} \Lambda_{3}.$$
Since $\Omega$ is the greatest common right inner divisor of $\Theta$ and $\Omega$, then $\Theta \Omega^{\ast}$ and $\Delta \Omega^{\ast}$ are right coprime. Hence $\Lambda_{3}$ is a constant unitary matrix. By $(\ref{shf57})$, we obtain the desired result.
\end{proof}

As we known, for a unique factorization domain $R$, if $a $ is a factor of $ bc$ and $a,b$ are coprime, then $a$ is a factor of $ c$. The following lemma shows that the similar property holds for two-sided inner functions.

\begin{lemma}\label{shly29}
Suppose $\theta$ is a nonconstant scalar inner function and $\Theta, \Omega $ are two-sided inner functions with values in  $\mathbb{M}_{n}$. We have the following results:
\begin{enumerate}
  \item If $\theta$ and $D(\Theta)$ are coprime, then $\theta$ is an inner factor of $\Theta \Omega$ if and only if $\theta$ is an inner factor of $\Omega$.
  \item If $\theta$ and $D(\Omega)$ are coprime, then $\theta$ is an inner factor of $\Theta \Omega$ if and only if $\theta$ is an inner factor of $\Theta$.
\end{enumerate}
\end{lemma}
\begin{proof}
The sufficiency of (1) and (2) is obvious.

For the necessity of (1), assume $\theta$ is an inner factor of $\Theta \Omega$, then there exists a two-sided inner function $\Delta$ such that
$$\Theta \Omega=\theta \Delta .$$
Since $\theta$ and $D(\Theta)$ are coprime, by Lemma \ref{shly1}, $\theta I_{n}$ and $\Theta$ are right coprime, by Lemma \ref{shly23},
$$\Theta \Omega H^{2}(\mathbb{C}^{n}) \subset \theta H^{2}(\mathbb{C}^{n})\cap \Theta H^{2}(\mathbb{C}^{n})=\theta \Theta H^{2}(\mathbb{C}^{n}). $$
Therefore, there exists a two-sided inner function $\Theta_{1}$ such that $\Theta \Omega=\theta \Theta \Theta_{1}$, which implies that $\theta$ is an inner factor of $\Omega$.

For the necessity of (2). By Lemma 4.12 in \cite{CHL2019}, we have $D(\widetilde{\Theta})=\widetilde{D(\Theta)}$. If $\theta$ is an inner factor of $\Theta \Omega$, then $\widetilde{\theta}$ is an inner factor of $ \widetilde{\Omega}\widetilde{\Theta}$. By (1), we obtain that $\widetilde{\theta}$ is an inner factor of $\widetilde{\Theta}$. Hence $\theta$ is an inner factor of $\Theta$.

\end{proof}

\begin{corollary}\label{shly28}
Suppose $F_{1},F_{2} \in H^{\infty}(\mathbb{M}_{n})$ and $\theta$ is a scalar inner function. Then the following statements hold:

\begin{enumerate}
  \item If $\theta I_{n}$ and $F_{1}$ are left coprime, and $\theta$ is an inner factor of $F_{1}F_{2}$, then $\theta$ is an inner factor of $F_{2}$.
  \item If $\theta I_{n}$ and $F_{2}$ are left coprime, and $\theta$ is an inner factor of $F_{1}F_{2}$, then $\theta$ is an inner factor of $F_{1}$.
\end{enumerate}
\end{corollary}

\begin{proposition}
Suppose $\Omega, \Theta, \Delta$ and $\Lambda$ are two-sided inner functions with values in  $\mathbb{M}_{n}$ and $\Theta \Delta= \Omega \Lambda$, then $$\Theta \Delta=\text{left-l.c.m.}\{\Theta, \Omega\}$$ if and only if $\Delta$ and $\Lambda$ are right coprime.
\end{proposition}
\begin{proof}
The necessity follows from Lemma \ref{shly32}.

For the sufficiency.
Let $\Gamma=\text{left-l.c.m.}\{\Theta, \Omega\}$, by Lemma \ref{shly32}
there exist two-sided inner functions $\Theta_{1}, \Omega_{1}$ that are right coprime such that
$$\Gamma =\Theta \Theta_{1} =\Omega \Omega_{1}.$$
Since $\Theta \Delta= \Omega \Lambda$ implies that $\Theta, \Omega$ are both left inner divisors of $\Theta \Delta$, then $\Gamma$ is a left inner divisor of $\Theta \Delta$. Thus there exists a two-sided inner function $\Gamma_{1}$ such that
$$\Theta \Delta= \Omega \Lambda =\Gamma \Gamma_{1}\Rightarrow \Delta =\Theta_{1} \Gamma_{1}, \Lambda=\Omega_{1} \Gamma_{1}.$$
Since $\Delta$ and $\Lambda$ are right coprime, then $\Gamma_{1}$ is a constant unitary matrix. Therefore $\Theta \Delta =\text{left-l.c.m.}\{\Theta, \Omega\}$.

\end{proof}

\subsection{The greatest common inner divisor of two-sided inner functions}

Recall that every matrix-valued two-sided inner function has a scalar inner multiple (see Definition \ref{sdef1}). However, it is worth pointing out that $D(\Delta)I_{n}=\Delta \Delta'$ does not imply that for any right factors $\Delta_{r}$ of $\Delta$, there exists a left factor $\Delta_{l}'$ of $\Delta'$, such that $\Delta_{r} \Delta_{l}'=\delta_{0}I_{n}$, where $\delta_{0}$ is a nonconstant inner factor of $D(\Delta)$.
\begin{example}
Suppose $u,v$ are coprime nonconstant scalar inner functions, let
\begin{equation}\nonumber
  \Theta=\left(
  \begin{array}{ccc}
    uv&0& 0\\
    0& u&0 \\
    0&0&v\\
  \end{array}
\right), ~\Theta'=\left(
  \begin{array}{ccc}
    1&0&0\\
    0&v& 0\\
    0&0&u\\
  \end{array}
\right),
\end{equation}
then $D(\Theta)=\Theta \Theta'=uv I_{n}$. It is clear that $\Theta$ has right factors
\begin{equation}\nonumber
  \Theta_{1}=\left(
  \begin{array}{ccc}
    u&0& 0\\
    0& 1&0 \\
    0&0&1\\
  \end{array}
\right), ~\Theta_{2}=\left(
  \begin{array}{ccc}
    u&0&0\\
    0&u& 0\\
    0&0&1\\
  \end{array}
\right),
\end{equation}
 and
  $$D(\Theta_{1})=D(\Theta_{2})=uI_{n},~~\Theta_{2}=\text{right-g.c.d.}\{uI_{n}, \Theta\}.$$
 It is not hard to check that there does not exist any left factors $\Theta_{1}'$ of $\Theta'$ such that $\Theta_{1}\Theta_{1}'=uI_{n}$, but there exists a left factor
 \begin{equation}\nonumber
\Theta_{2}'=\left(
  \begin{array}{ccc}
    1&0&0\\
    0&1& 0\\
    0&0&u\\
  \end{array}
\right),
\end{equation}
of $\Theta'$
such that $\Theta_{2} \Theta_{2}'=uI_{n}$.
\end{example}

The above simple example inspired us to get the following crucial lemma, which will be useful for the proof of our main results.
\begin{lemma}\label{shly25}
Suppose $\Theta$ and $ \Delta$ are two-sided inner functions with values in  $\mathbb{M}_{n}$ without nonconstant scalar inner factors. If $\Theta \Delta$ has a nonconstant scalar inner factor $\theta$, let $\Theta_{r}=\text{right-g.c.d.}\{\theta I_{n}, \Theta\}$, then there exists $\Delta_{l}$ which is a left inner divisor of $\Delta$, such that $\Theta_{r} \Delta_{l}=\theta I_{n}$. Moreover,
$$D(\Theta_{r})=D(\Delta_{l})=\theta.$$
\end{lemma}
\begin{proof}
If $\theta$ and $D(\Theta)$(or $D(\Delta)$) are  coprime,
by Lemma \ref{shly29}, $\theta$ is an inner divisor of $\Delta$(or $\Theta$). This contradicts the fact that $\Theta$ and $\Delta$ have no nonconstant scalar inner factors. Therefore, $\theta$ and $D(\Theta)$(or $D(\Delta)$) are not coprime.
Since $\Theta \Delta$ has a nonconstant scalar inner factor $\theta$, then there exists a two-sided inner function $\Omega$ such that
$\Theta \Delta=\theta \Omega$. This yields that $\Theta, \theta I_{n}$ are left inner divisor of $\Theta \Delta$, thus
\begin{equation}\label{shf62}
  \Theta \Delta H^{2}(\mathbb{C}^{n}) \subset \theta H^{2}(\mathbb{C}^{n}) \cap \Theta H^{2}(\mathbb{C}^{n}).
\end{equation}
By Lemma \ref{shly23}, $\text{left-l.c.m.}\{\theta I_{n}, \Theta\}=\theta \Theta \Theta_{r}^{\ast}$. Combining this with $(\ref{shf62})$, there exists a two-sided inner function $\Lambda$ such that $\Theta \Delta=\theta \Theta \Theta_{r}^{\ast} \Lambda$, which implies that $\theta \Theta_{r}^{\ast} $ is a left inner divisor of $\Delta$. Let $\Delta_{l}=\theta \Theta_{r}^{\ast} $, then $\Theta_{r} \Delta_{l} =\theta I_{n}$.
Since $\Theta, \Delta$ have no nonconstant scalar inner factors, then so do $\Theta_{r}, \Delta_{l}$. By Lemma \ref{shly9},
$$D(\Theta_{r})=D(\Delta_{l})=\theta.$$
\end{proof}
\begin{remark}\label{shre3}
It is clear that $\Theta \Delta$ has a nonconstant scalar inner factor $\theta$ if and only if $ \widetilde{\Delta }\widetilde{\Theta}$ has a nonconstant scalar inner factor $\widetilde{\theta}$. Let
$$\Delta_{l}=\text{left-g.c.d.}\{\theta I_{n}, \Delta\}, $$
by a similar argument to the proof of Lemma \ref{shly25}, there exists $\Theta_{r}$ is a right factor of $\Theta$ such that
$\Theta_{r} \Delta_{l}=\theta I_{n}$.
\end{remark}

%

\begin{corollary}\label{shly3}
Suppose  $\Theta$ is a two-sided inner function with values in  $\mathbb{M}_{n}$ and $D(\Theta)=\theta_{1} \theta_{2}$, where $\theta_{i}$ are nonconstant scalar inner functions. Let
$$\Theta_{1}=\text{left-g.c.d.}\{\theta_{1}I_{n}, \Theta\},$$
 then $D(\Theta_{1}) = \theta_{1}$.
\end{corollary}
\begin{proof}
If $\Theta$ has scalar inner factor, let $\Theta=G(\Theta) \Lambda$, where $G(\Lambda)=1$. Let $\omega=\text{g.c.d.}\{ \theta_{1},G(\Theta)\}$,  $\omega_{1}=\theta_{1}\omega^{\ast}$ and $\Lambda_{1}=\text{left-g.c.d.}\{\omega_{1}I_{n}, \Lambda\}$.  Then the conclusion is equivalent to  $D(\Lambda_{1}) = \omega_{1}$.
So without lost of generality, we can assume that $\Theta$ has no nonconstant scalar inner factors, i.e., $G(\Theta)=1$.
Write $D(\Theta)I_{n}=\Theta \Theta'$, then the conclusion follows by Lemma \ref{shly25}(take $\Delta=\Theta', \theta=\theta_{1}$).



\end{proof}

\subsection{The proof of Theorem \ref{shly10}}
Before we prove Theorem \ref{shly10}, we need two lemmas.
\begin{lemma}\label{shly4}
Suppose  $\Theta$ is a two-sided inner function with values in  $\mathbb{M}_{n}$ and $\theta$ is a nonconstant scalar inner function.  Let $\delta=\text{g.c.d}\{\theta, D(\Theta)\}$, then
$$\text{left-g.c.d.}\{\theta I_{n}, \Theta\}=\text{left-g.c.d.}\{\delta I_{n}, \Theta\}.$$
\end{lemma}
\begin{proof}
Let
$$\Theta_{1}=\text{left-g.c.d.}\{\delta I_{n}, \Theta\}~\text{and}~\Theta_{2}=\text{left-g.c.d.}\{\theta I_{n}, \Theta\}, $$
then $\Theta_{1}$ is a left inner factor of $\Theta_{2}$. Next, we show that $$\Theta_{1} = \Theta_{2}.$$
By Lemma \ref{shly2},
$D(\Theta_{1})$ is a factor of $D(\Theta_{2})$, and $D(\Theta_{2})$ is a common inner factor of $\theta$ and $D(\Theta)$, which implies that $D(\Theta_{2})$ is a factor of $\delta$. By Corollary \ref{shly3}, $D(\Theta_{1})=\delta$, then $D(\Theta_{1})=D(\Theta_{2})$. Since $\Theta_{2}$ is the greatest common left inner divisor of $\theta$ and $\Theta$, then $\Theta_{2}^{\ast} \theta$ and $\Theta_{2}^{\ast} \Theta$ are left coprime. Since $\Theta_{2}^{\ast} D(\Theta_{2})$ is a left inner factor of $\Theta_{2}^{\ast} \theta$, then $\Theta_{2}^{\ast} D(\Theta_{2})$ and $\Theta_{2}^{\ast} \Theta$ are left coprime. Therefore,
$$\Theta_{2}= \text{left-g.c.d.}\{D(\Theta_{2}) I_{n}, \Theta\}=\text{left-g.c.d.}\{\delta I_{n}, \Theta\}=\Theta_{1}.$$

\end{proof}


\begin{lemma}\label{shly6}
Suppose $\theta$ is a nonconstant scalar inner function and $\Theta, \Delta $ are two-sided inner functions with values in  $\mathbb{M}_{n}$. If $\theta$ and $D(\Delta)$ are  coprime, then
\begin{equation}\label{shf1}
 \text{left-g.c.d}\{\theta I_{n} , \Theta \Delta\}= \text{left-g.c.d}\{\theta I_{n}, \Theta\}.
\end{equation}
\end{lemma}
\begin{proof}
  By Lemma \ref{shly2}, there exist scalar inner functions $\theta_{1}$ and $\theta_{2}$ such that
$$D(\Theta \Delta) =\theta_{1} D(\Delta)=D(\Theta)\theta_{2},$$ and $\theta_{1},\theta_{2}$ are factors of $D(\Theta)$ and $D(\Delta)$, respectively.
Since $\theta$ and $D(\Delta)$ are coprime, then
\begin{equation}\label{shf3}
  \text{g.c.d.}\{\theta, D(\Theta \Delta)\}=\text{g.c.d.}\{\theta, \theta_{1}D(\Delta)\}=\text{g.c.d.}\{\theta, \theta_{1}\}=\text{g.c.d.}\{\theta, D(\Theta)\}.
\end{equation}
Let $\delta=\text{g.c.d.}\{\theta, D(\Theta)\}$. Using Lemma \ref{shly4} and $(\ref{shf3})$, we have
\begin{align*}
 &\text{left-g.c.d.}\{\theta I_{n}, \Theta \Delta\} = \text{left-g.c.d.}\{\delta I_{n}, \Theta \Delta\},\\
&\text{left-g.c.d.}\{\theta I_{n}, \Theta \} = \text{left-g.c.d.}\{\delta I_{n}, \Theta \}.
\end{align*}
Let
\begin{equation}\label{shf7}
  \Theta_{1}=\text{left-g.c.d.}\{\delta I_{n}, \Theta \Delta\}, ~\Theta_{2}=\text{left-g.c.d.}\{\delta I_{n}, \Theta\}.
\end{equation}
Next, we will prove that $\Theta_{1}=\Theta_{2}$.
By Corollary \ref{shly3}, $D(\Theta_{1})=D(\Theta_{2})=\delta$. It is clear that $\Theta_{2}$ is a left inner divisor of $\Theta_{1}$. Let $$\Theta_{1}=\Theta_{2}\Theta_{0},$$
then $\Theta_{0}$ is a left inner divisor of $\Theta_{2}^{\ast} \delta$.
We will now show that $\Theta_{0}$ is a constant unitary matrix.

By Lemma \ref{shly2}, $D(\Theta_{0})$ is a factor of $\delta$. By $(\ref{shf7})$, there exist two-sided inner functions $\Theta_{3}$ and $\Omega$ such that
$$\Theta=\Theta_{2} \Theta_{3}, \Theta \Delta= \Theta_{1}\Omega,$$
which gives us
$$ \Theta_{2} \Theta_{3} \Delta=\Theta \Delta= \Theta_{1}\Omega=\Theta_{2}\Theta_{0}\Omega.$$
Hence $\Theta_{3} \Delta=\Theta_{0}\Omega$. Write $D(\Theta_{0})I_{n}=\Theta_{0}\Theta_{0}'$, then we get
$$\Theta_{0}' \Theta_{3} \Delta=\Theta_{0}'\Theta_{0}\Omega=D(\Theta_{0}) \Omega.$$
Since $\theta$ and $D(\Delta)$ are coprime and $D(\Theta_{0})$ is a factor of $\delta=\text{g.c.d.}\{\theta, D(\Theta)\}$, then $D(\Theta_{0})$ and $D(\Delta)$ are also coprime. By Lemma \ref{shly29}, $D(\Theta_{0})$ is a scalar inner factor of $\Theta_{0}' \Theta_{3}$, then exists a two-sided inner function $\Lambda$ such that
$$\Theta_{0}' \Theta_{3}= D(\Theta_{0}) \Lambda=\Theta_{0}' \Theta_{0} \Lambda . $$
Thus $\Theta_{0}$ is a left inner factor of $\Theta_{3}=\Theta_{2}^{\ast} \Theta$. Hence $\Theta_{0}$ is a common left inner divisor of $\Theta_{2}^{\ast} \delta$ and $\Theta_{2}^{\ast} \Theta$.
Since $\Theta_{2}$ is the greatest common left inner divisor of $\delta$ and $\Theta$, we conclude that $\Theta_{0}$ is a constant unitary matrix. Therefore, $\Theta_{1}=\Theta_{2}$. 

\end{proof}

%

\begin{theorem}\label{shly10}
Suppose $\Delta_{i}, i=1,2$ are two-sided inner functions with values in  $\mathbb{M}_{n}$ without nonconstant scalar inner factors and $D(\Delta_{i})=\delta, i=1,2$, where $\delta$ is a scalar inner function. If $\theta$ is a scalar inner function and
\begin{equation}\label{shf9}
  \text{left-g.c.d}\{\theta I_{n}, \Delta_{1} \Delta_{2}\}=\text{left-g.c.d}\{\theta I_{n}, \delta\Delta_{1}  \}=\Omega,
\end{equation}
then $\Omega$ has the form $\omega I_{n}$ for some scalar-valued inner function $\omega$. Moreover,
$$\omega =\text{g.c.d.}\{\theta, \delta\},$$ and $\omega^{\ast} \theta$ and $\delta $ are coprime.

\end{theorem}
\begin{proof}
Let $\delta_{0}=\text{g.c.d.}\{\theta, \delta\}$.
The following proofs are divided into two cases: Case I, $\delta_{0} =1$; Case II, $\delta_{0}\neq 1$.

For Case I. By Lemma \ref{shly2}, $D( \Delta_{1} \Delta_{2})$ is a factor of $\delta^{2}$. Since $\theta$ and $\delta$ are coprime, then $\theta$ and $\delta^{2}$ are coprime. By Lemma \ref{shly1}, $\theta I_{n}$ and $ \Delta_{1} \Delta_{2}$ are left coprime, then $\Omega=I_{n}(=\delta_{0}I_{n})$.

For Case II. Since $\delta_{0} \neq 1$, let $\delta_{1}=\text{g.c.d.}\{\delta_{0}^{\ast}\theta, \delta\}$. Next, we will show that $\delta_{1}=1$ and $\Omega=\delta_{0}I_{n}$.

Since $\delta_{0}^{\ast} \theta$ and $\delta_{0}^{\ast} \delta$ are coprime, then $\delta_{1}$ is a factor of $\delta_{0}$, and $\delta_{1}$ and $\delta_{0}^{\ast} \delta$ are coprime. It is clear that $D(\delta\Delta_{1})=\delta^{2}$, then
$$\text{g.c.d.}\{\theta, D( \delta \Delta_{1})\}=\text{g.c.d.}\{\theta, \delta^{2}\}=\delta_{0} \delta_{1}.$$
By Lemma \ref{shly4},
\begin{equation}\label{shf18}
  \begin{split}
   \text{left-g.c.d.}\{\theta I_{n}, \delta \Delta_{1}\}
   &=\text{left-g.c.d.}\{\delta_{0} \delta_{1} I_{n},\delta \Delta_{1}\}\\
   &=\delta_{0}\text{left-g.c.d.}\{ \delta_{1} I_{n}, \Delta_{1}(\delta^{\ast}_{0}\delta)\}\\
   &=\delta_{0}\text{left-g.c.d.}\{ \delta_{1} I_{n}, \Delta_{1}\}(\text{$\delta_{1}$ and $\delta_{0}^{\ast} \delta$ are coprime}).
  \end{split}
\end{equation}
By $(\ref{shf9})$ and $(\ref{shf18})$,
$\delta_{0}$ is a scalar inner factor of $\Delta_{1}\Delta_{2} $. By Lemma  \ref{shly25}, there exist $\Delta_{1l}, \Delta_{1r}, \Delta_{2l}, \Delta_{2r}$ such that
 $$\Delta_{1}=\Delta_{1l}\Delta_{1r},~~\Delta_{2}=\Delta_{2l} \Delta_{2r}~\text{and}~\Delta_{1r}\Delta_{2l}=\delta_{0}I_{n}.$$
Since $\Delta_{1}$ and $\Delta_{2}$ have no nonconstant scalar inner factors, then $\Delta_{1r},\Delta_{2l}$ have no nonconstant scalar inner factors. By Lemma \ref{shly9},
$$D(\Delta_{1r})=D(\Delta_{2l})=\delta_{0}.$$

Let $\Delta_{2l}=\text{left-g.c.d.}\{\delta_{0} I_{n},\Delta_{2}\}$ and $D(\Delta_{2})I_{n}=\Delta_{2}' \Delta_{2}$  by Remark \ref{shre3}, there exists   a right inner factor $ \Delta_{2r}'$ of $\Delta_{2}'$ such that $ \Delta_{2r}'\Delta_{2l}=\delta_{0}I_{n}$. This shows that $(\Delta_{2}' (\Delta_{2r}')^{\ast} )\Delta_{2r}=\delta_{0}^{\ast} \delta I_{n}$.
By Lemma \ref{shly7}, $\Delta_{2},\Delta_{2}'$ have no nonconstant scalar inner factor, by Lemma \ref{shly9}, $D(\Delta_{2r})=\delta_{0}^{\ast} \delta$.
 Since $\delta_{0}^{\ast} \theta$ and $\delta_{0}^{\ast} \delta$ are coprime, and $ D(\Delta_{2r})=\delta_{0}^{\ast} \delta$, then $\delta_{0}^{\ast} \theta$ and $\Delta_{2r}$ are left coprime. By Lemma \ref{shly6}, we get
 \begin{equation}\nonumber
   \text{left-g.c.d.}\{\theta I_{n}, \Delta_{1}\Delta_{2} \}
 =\delta_{0}~\text{left-g.c.d.}\{\delta_{0}^{\ast}\theta I_{n}, \Delta_{1l}\Delta_{2r} \}=\delta_{0}~\text{left-g.c.d.}\{\delta_{1} I_{n}, \Delta_{1l} \}.
 \end{equation}
By $(\ref{shf9})$, $(\ref{shf18})$,
\begin{equation}\label{shf11}
\begin{split}
  \text{left-g.c.d}\{\theta I_{n}, \delta \Delta_{1}  \}=\delta_{0}~\text{left-g.c.d.}\{ \delta_{1} I_{n}, \Delta_{1}\}.
\end{split}
\end{equation}
Next, we show that $(\ref{shf9})$ holds if and only $\delta_{1}=1$. To this, we only need to show that
if $\delta_{1}\neq 1$, then
$$\text{left-g.c.d.} \{\delta_{1} I_{n}, \Delta_{1l} \} \neq \text{left-g.c.d.}\{ \delta_{1} I_{n}, \Delta_{1}\}.$$
Let $D(\Delta_{1})I_{n}=\Delta_{1}'\Delta_{1}$ and
$$\Omega_{1}=\text{left-g.c.d}\{\delta_{1} I_{n}, \Delta_{1} \},~\Omega_{2}=\text{left-g.c.d}\{\delta_{1} I_{n}, \Delta_{1l} \}.$$
By Remark \ref{shre3}, there exists a right factor $\Omega_{1}'$ of $\Delta_{1}'$ such that $\Omega_{1}' \Omega_{1}=\delta_{1}I_{n}$.
However, if  $ \Omega_{1}'\Omega_{2}=\delta_{1}I_{n}$, we get
 $$ \delta I_{n}=\Delta_{1}'(\Omega_{1}')^{\ast} \Omega_{1}'\Omega_{2} \Omega_{2}^{\ast} \Delta_{1l}\Delta_{1r}=\delta_{1}(\Delta_{1}'(\Omega_{1}')^{\ast}) ( \Omega_{2}^{\ast} \Delta_{1l})\Delta_{1r}. $$
Since $D(\Delta_{1r})=\delta_{0}$, the above equation implies that $\delta_{1}$ is a factor of $\delta_{0}^{\ast} \delta$. By the definition of $\delta_{0}$ and $\delta_{1}$ , we have $g.c.d.\{\delta_{0}^{\ast} \theta, \delta_{0}^{\ast} \delta\}=1$ and $\delta_1$ is a factor of $\delta_{0}^*\theta$ , thus we get that $\delta_{1}$ and $\delta_{0}^{\ast} \delta$ are coprime.
Therefore, $\delta_{1}=1$. This contradicts the assumption  $\delta_{1}\neq 1$. Hence $\Omega_{1}'\Omega_{2}\neq \delta_{1} I_{n}$. This means $\Omega_{1}\neq \Omega_{2}$.

Combining the above arguments, we obtain that $(\ref{shf9})$ holds if and only if $\delta_{1}=1$ and $\Omega=\delta_{0}I_{n}$.

\end{proof}

\begin{corollary}\label{shly17}
Suppose $\Delta_{i}, i=1,2$ are two-sided inner functions with values in  $\mathbb{M}_{n}$ without nonconstant scalar inner factors and $D(\Delta_{i})=\delta, i=1,2$. If $\theta,\omega$ are scalar inner functions and
\begin{equation}\nonumber
 \text{left-g.c.d}\{\theta I_{n}, \Delta_{1} \Delta_{2}\}=\text{left-g.c.d}\{\theta I_{n}, \omega \delta \Delta_{1}  \}=\Omega,
\end{equation}
then
$\Omega =\text{g.c.d.}\{\theta,  \omega \delta \} I_{n}$, and $\Omega^{\ast} \theta$ and $\delta$ are coprime.

\end{corollary}
\begin{proof}
Let $\delta_{0}=\text{g.c.d.}\{\theta,  \omega \delta \}$, then $\delta_{0}$ is the factor of $\delta$ because
$$\text{left-g.c.d}\{\theta I_{n}, \Delta_{1} \Delta_{2}\}=\text{left-g.c.d}\{\theta I_{n}, \omega\delta\Delta_{1} \}.$$
Let $\delta_{1}=\text{g.c.d.}\{\delta^{\ast}_{0}\theta,   \delta \}$ then $\delta_{1}=\text{g.c.d.}\{\delta^{\ast}_{0}\theta,  \delta_{0} \}$ because $\delta_{0}^{\ast} \theta$ and $\delta_{0}^{\ast} \delta \omega$ are coprime.
The rest proof follows from the proof of Theorem \ref{shly10}.
\end{proof}


\subsection{``Weak" commutativity of two-sided inner functions}


Suppose $\Theta, \Delta$ are two-sided inner functions with values in $\mathbb{M}_{n}$, then the commutative property
$$\Theta \Delta = \Delta \Theta$$ does not hold for $n> 1$. Therefore, it is natural to ask whether there exists ``weakly" commutative property.
\begin{question}\label{shq1}
Whether there exists a two-sided inner function $\Omega$ such that
\begin{equation}\label{shf45}
  \Theta \Delta= \Delta \Omega~\text{or}~\Theta \Delta=\Omega \Theta?
\end{equation}
\end{question}
 However, the following example demonstrates that $(\ref{shf45})$ doesn't hold in general. If $(\ref{shf45})$ were to hold, then we would have either

$$\Omega=\Delta^{\ast}\Theta \Delta ~\text{or}~\Omega=\Theta \Delta \Theta^{\ast}.$$
\begin{example}\label{shex1}
Let
\begin{equation}\label{shf46}
\Theta=\frac{\sqrt{2}}{2}\left(
  \begin{array}{cc}
    z^{3} & -z^{5}\\
    1 & z^{2} \\
  \end{array}
\right),
~\Delta=\left(
\begin{array}{cc}
    0 & 1 \\
     z& 0 \\
  \end{array}
\right),
\end{equation}
then
\begin{equation}\nonumber
\Delta^{\ast}\Theta \Delta=\frac{\sqrt{2}}{2}\left(
  \begin{array}{cc}
    z^{2} & \overline{z}\\
    -z^{6} & z^{3} \\
  \end{array}
\right),
~\Theta \Delta \Theta^{\ast}=\frac{1}{2}\left(
\begin{array}{cc}
    -z^{3}-\overline{z}^{2} & -z^{6}+z \\
     1-\overline{z}^{5}& z^{3}-\overline{z}^{2} \\
  \end{array}
\right).
\end{equation}
Since both $\Delta^{\ast}\Theta \Delta$ and $\Theta \Delta \Theta^{\ast}$ are not analytic, then $(\ref{shf45})$ doesn't hold. This implies that in the noncommunicative setting, $\Delta$ is not the left inner factor of $\Theta \Delta$ and $\Theta$ is not the right inner factor of $\Theta \Delta$. This is in contrast to the commutative setting where such factorizations exist.
\end{example}
In \cite{CHL2019}, the authors show that if $A$ and $B$ are left coprime and $A$ and $C$ are left coprime, where $A,B,C \in H^{\infty}(\mathbb{M}_{n})$ and $AB=BA$, then $A$ and $BC$ are also left coprime. However, if we remove the condition $AB=BA$, the result will no longer hold(see Example \ref{shex2}).
\begin{example}\label{shex2}
Let $A=\Delta$ and $B=C=\Theta$ in $(\ref{shf46})$. Since $\det(A)=-z$, then $A$ has no nontrivial divisors. It is clear that
\begin{equation}\nonumber
AB=\frac{\sqrt{2}}{2}\left(
  \begin{array}{cc}
    1 & z^{2}\\
   z^{4} & -z^{6} \\
  \end{array}
\right)\neq \frac{\sqrt{2}}{2}\left(
\begin{array}{cc}
    -z^{6} & -z^{3} \\
    z^{3}& 1 \\
  \end{array}
\right)=BA
\end{equation}
and
\begin{equation}\nonumber
A^{\ast}B=\frac{\sqrt{2}}{2}\left(
  \begin{array}{cc}
    \overline{z} & z\\
   z^{3} & -z^{5} \\
  \end{array}
\right),~ A^{\ast} BC=\frac{1}{2}\left(
\begin{array}{cc}
    z^{2}+z & -z^{4}+z^{3} \\
    z^{6}-z^{5}& -z^{8}-z^{7} \\
  \end{array}
\right).
\end{equation}
Since $A^{\ast} B$ is not analytic and $A$ has no nontrivial left inner divisor, then $A $ and $B$ are left coprime and $A $ and $C$ are left coprime. However, $A^{\ast} BC$ is analytic, which shows that $A$ is the left divisor of $BC$.

\end{example}
The following example shows that the unique decomposition of scalar inner functions does not hold for matrix-valued inner functions.
\begin{example}\label{shex3}
\begin{equation}\nonumber
  \left(
  \begin{array}{cc}
    z^{3} & -z^{5}\\
    1 & z^{2} \\
  \end{array}
\right)= \left(
  \begin{array}{cc}
    z &0\\
   0 & 1 \\
  \end{array}
\right)\left(
  \begin{array}{cc}
    z^{2} & -z^{4}\\
    1 & z^{2} \\
  \end{array}
\right)=\left(
  \begin{array}{cc}
    -z^{4} & z^{3}\\
    z & 1 \\
  \end{array}
\right)\left(
  \begin{array}{cc}
    0 & z\\
    1 & 0 \\
  \end{array}
\right).
\end{equation}
\end{example}

Example \ref{shex1}, Example \ref{shex2} and Example \ref{shex3} demonstrate some differences between the commutative setting and the noncommutative setting. Based on these observations, we modify Question \ref{shq1} as follows:
\begin{question}\label{shq3}
Suppose $\Theta$ and $ \Delta$ are two-sided inner functions with values in $\mathbb{M}_{n}$, whether there exist $\Theta_{1}, \Delta_{1}$ such that
\begin{equation}\label{shf50}
  \Theta \Delta=\Delta_{1} \Theta_{1},
\end{equation}
and $\Delta_{1}$ preserves some ``suitable" properties of $\Delta$? 
\end{question}

\begin{remark}\label{shre1}
Suppose $F\in H^{2}(\mathbb{M}_{n})$, then $F$ has the following inner-outer factorization of the form:
$$F=F_{li}F_{lo}=F_{ro}F_{ri},$$
where $F_{li}, F_{ri}$ are inner functions and $F_{lo}, F_{ro}$ are outer functions.
By the proof of Theorem 4.17 in \cite{CHL2019}, we have $D(F_{li})=D(F_{ri})$. By Helson-Lowdenslager Theorem in \cite{Nikolskii1986}, $\det(F_{li})=\det(F_{ri})$, where $\det$ denotes the determinant of a matrix.

\end{remark}
Enlightened by the inner-outer factorization, we study the decomposition has the form in $(\ref{shf50})$ with $D(\Delta_{1})=D(\Delta)$ or $\det(\Delta_{1}) =\det(\Delta)$.
By Example \ref{shex3}, we have
\begin{equation}\nonumber
D\Big(\left(
  \begin{array}{cc}
    z &0\\
   0 & 1 \\
  \end{array}
\right)\Big)=D\Big(\left(
  \begin{array}{cc}
    0 & z\\
    1 & 0 \\
  \end{array}
\right)\Big).
\end{equation}
Therefore, Example \ref{shex3} gives an affirmative answer to Question \ref{shq3}. And it inspires us to obtain the following results.

\begin{proposition}\label{shly14}
Suppose $\Theta$ and $ \Delta$ are two-sided inner functions with values in  $\mathbb{M}_{n}$, then the following statements hold:
\begin{description}
  \item[(a)] there exist two-sided inner functions $\Theta_{1}, \Delta_{1}$ satisfying $D(\Theta_{1})=D(\Theta)$ and $D(\Delta \Theta)=D(\Theta_{1})D(\Delta_{1})$ such that
$$\Delta \Theta= \Theta_{1}\Delta_{1}.$$
  \item[(b)] there exist two-sided inner functions $\Theta_{2}, \Delta_{2}$ satisfying $D(\Delta_{2})=D(\Delta)$ and $D(\Delta \Theta)=D(\Theta_{2})D(\Delta_{2})$ such that
$$\Delta \Theta=\Theta_{2}\Delta_{2}.$$
\item[(c)]$G(\Theta_{1})=G(\Theta)$;
  \item[(d)]  if $D(\Delta)$ and $D(\Theta)$ are coprime, then $\det(\Theta_{1})=\det(\Theta)=\det(\Theta_{2})$.
\end{description}
\end{proposition}

\begin{proof}
For (a),
by Lemma \ref{shly2}, $D(\Theta)$ is a factor of $D(\Delta \Theta)$. Let
\begin{equation}\label{shf70}
  \Theta_{1}=\text{left-g.c.d}\{D(\Theta) I_{n}, \Delta \Theta\},
\end{equation}
then there exists $\Delta_{1}$ such that $\Delta\Theta=\Theta_{1}\Delta_{1}$.
By Corollary \ref{shly3}, $D(\Theta_{1})=D(\Theta)$.
Write $D(\Delta \Theta)I_{n}=\Omega (\Delta \Theta)$, using Remark \ref{shre3}, there exists a right factor $\Omega_{r}$ of $\Omega$ such that $\Omega_{r} \Theta_{1}=D(\Theta) I_{n}$. Hence
\begin{equation}\label{shf75}
  D(\Delta \Theta) I_{n}=\Omega (\Delta \Theta)=\Omega \Theta_{1}\Delta_{1}=D(\Theta_{1}) (\Omega \Omega_{r}^{\ast})\Delta_{1} .
\end{equation}
By Lemma \ref{shly7}, $\Omega$ has no nonconstant scalar inner factors, this implies that $ \Omega \Omega_{r}^{\ast}$ also has no nonconstant scalar inner factors. By Lemma \ref{shly9}, $ (\Omega \Omega_{r}^{\ast})\Delta_{1}=D(\Delta_{1})I_{n}$. Therefore, $(\ref{shf75})$ yields $D(\Delta \Theta)=D(\Theta_{1})D(\Delta_{1})$.

The proof of (b) is similar to (a).

(c) follows from the fact that scalar inner factors commutate with each other.

For (d), let $D(\Theta)I_{n}=\Theta \Theta_{0}$, then $\det(\Theta)$ is a factor of $D(\Theta)^{n}$. Since $\det(\Theta)=\Theta adj(\Theta)$,where $adj(\Theta)$ denotes the adjugate matrix of $\Theta$, then $D(\Theta)$ is a factor of $\det(\Theta)$. Let $\det(\Theta)=D(\Theta)u$, then $u$ is a factor of $D(\Theta)^{n-1}$.
Similarly, we obtain that $\det(\Theta_{1})=D(\Theta)u'$, where $u'$ is a factor of $D(\Theta)^{n-1}$. Since $D(\Delta)$ and $D(\Theta)$ are coprime, by Lemma C.13 in \cite{CHL2021}, $D(\Theta)$ and $\det(\Delta)$ are coprime. Since
$$\det(\Theta)\det(\Delta)=\det(\Theta_{1})\det(\Delta_{1}),$$
then $\det(\Theta_{1})$ is a factor of $\det(\Theta)$. If $\det(\Theta_{1})\neq  \det(\Theta)$, then there exists a nonconstant scalar inner function $\omega$ which is an inner factor of $D(\Theta)^{n}$ such that $$\det(\Theta)=\omega \det(\Theta_{1}) ~\text{and}~ \det(\Delta_{1})=\omega\det(\Delta).$$
Hence $\omega$ and $D(\Theta)$ are not coprime. Let $\omega_{0}=\text{g.c.d.}\{\omega, D(\Theta)\} \neq 1$, then
\begin{equation}\label{shf69}
  \text{g.c.d.}\{D(\Theta), \det(\Delta_{1})\}=\text{g.c.d.}\{D(\Theta), \omega\det(\Delta)\}=\omega_{0}.
\end{equation}
By Lemma \ref{shly1} and $(\ref{shf69})$, $D(\Theta)$ and $D(\Delta_{1})$ are not coprime, then $D(\Theta)$ and $\Delta_{1}$ are not left coprime.
Let $$\omega_{1}=\text{g.c.d.}\{D(\Theta), D(\Delta_{1})\} \neq 1~\text{and}~
 \Theta_{11}=\text{left-g.c.d.}\{D(\Theta) I_{n}, \Delta_{1}\}.$$
By Corollary \ref{shly3}, $D(\Theta_{11})=\omega_{1}$.

If $D(\Theta_{1} \Theta_{11})=D(\Theta)$,  $(\ref{shf70})$ shows that $\Theta_{11}$ is a constant unitary matrix. This contradicts the fact that $\omega_{1} \neq 1$.

If $D(\Theta_{1} \Theta_{11})\neq D(\Theta)$, by Lemma \ref{shly2}, there exists a nonconstant scalar inner function $\omega_{2}$ which is an inner divisor of $\omega_{1}$ such that $D(\Theta_{1} \Theta_{11})= \omega_{2}D(\Theta)$. Since
$$D(\Theta)\delta=D(\Delta \Theta)=D(\Theta_{1}\Theta_{11}(\Theta_{11}^{\ast} \Delta_{1}))=D(\Theta_{1}\Theta_{11}) \delta'$$
for some scalar inner functions $\delta$ and $\delta'$, then $\delta=\omega_{2} \delta'$. Since $D(\Theta)$ and $D(\Delta)$ are coprime, by Lemma \ref{shly2}, $D(\Delta)=\delta$, then $\omega_{2}$ is a common divisor of $D(\Delta)$ and $D(\Theta)$. This is a contradiction.

Combining the above arguments, we obtain $\det(\Theta_{1})= \det(\Theta)$. Similarly, we can show that $\det(\Delta_{2})=\det(\Delta)$, this yields that $\det(\Theta_{2})=\det(\Theta)$.
 \end{proof}

\begin{lemma}\label{shly13}
Suppose $\theta$ is a nonconstant scalar inner function and $\Theta, \Delta $ are two-sided inner functions with values in  $\mathbb{M}_{n}$. If $\theta$ and $D(\Delta)$ are  coprime, then there exists a two-sided inner function $\Theta_{1} $ such that $D(\Theta_{1})=D(\Theta)$ and
\begin{equation}\label{shf22}
  \text{left-g.c.d}\{\theta I_{n}, \Delta\Theta\}=\text{left-g.c.d}\{\theta I_{n}, \Theta_{1}\}.
\end{equation}
\end{lemma}
\begin{proof}
By Proposition \ref{shly14}, there exist two two-sided inner functions $\Theta_{1}, \Delta_{1} $ such that
\begin{equation}\label{shf23}
  \Delta \Theta=\Theta_{1} \Delta_{1}~ \text{and} ~D(\Theta_{1})=D(\Theta), D(\Delta \Theta)=D(\Theta_{1})D(\Delta_{1}),
\end{equation}
where $\Theta_{1}=\text{left-g.c.d.}\{D(\Theta) I_{n}, \Delta \Theta\}$.
By lemma \ref{shly2}, $D(\Delta\Theta)$ is the inner factor of $D(\Theta)D(\Delta)$, this means $ D(\Delta_{1})$ is a factor of $D(\Delta)$. Since $\theta$ and $D(\Delta)$ are  coprime, then $\theta$ and $D(\Delta_{1})$ are  coprime. By Lemma \ref{shly6},
 \begin{equation}\label{shf24}
  \text{left-g.c.d}\{\theta I_{n}, \Delta\Theta\}=\text{left-g.c.d}\{\theta I_{n}, \Theta_{1} \Delta_{1}\}=\text{left-g.c.d}\{\theta I_{n}, \Theta_{1}\}.
\end{equation}

%
%
%
%

\end{proof}

\begin{proposition}\label{shly16}
Suppose $\Theta, \Omega$ and $ \Delta$ are two-sided inner functions with values in  $\mathbb{M}_{n}$, and $\Phi \in H^{2}(\mathbb{M}_{n})$  and $\theta$ is a scalar inner function.
Then the following statements hold:
\begin{enumerate}
  \item If $\theta$ and $D(\Delta)$ are coprime, then there exists $\Theta_{1}$ such that $D(\Theta_{1})=D(\Theta)$ and
$$\text{left-g.c.d}\{\theta I_{n}, \Omega \Delta\Theta\}=\text{left-g.c.d}\{\theta I_{n},  \Omega\Theta_{1}\}.$$
  \item If $\theta$ and $\Phi$ are left coprime, then there exists $\Theta_{1}$ such that $D(\Theta_{1})=D(\Theta)$ and
$$\text{left-g.c.d}\{\theta I_{n}, \Omega \Phi \Theta\}=\text{left-g.c.d}\{\theta I_{n},  \Omega\Theta_{1}\}.$$
\end{enumerate}

\end{proposition}
\begin{proof}
For (1),
by $(\ref{shf23})$ in the proof of Lemma \ref{shly13}, we have $\Omega \Delta\Theta=\Omega\Theta_{1} \Delta_{1}$ and $D(\Theta_{1})=D(\Theta)$, and the argument in Lemma \ref{shly13} shows that $\theta$ and $D(\Delta_{1})$ are coprime. By Lemma \ref{shly6}, we obtain the desired result.

For (2), let $\Phi=\Phi_{i}\Phi_{o}$ be the inner-outer factorization of $\Phi$, and $h=\Phi_{o}\Theta$, then $h$ has the inner-outer factorization of the form $h=h_{i}h_{o}$.
By the proof of Theorem 4.17 in \cite{CHL2019}, we have $D(h_{i})=D(\Theta)$.
By Lemma \ref{shly6},
$$\text{left-g.c.d}\{\theta I_{n}, \Omega \Phi \Theta\}=\text{left-g.c.d}\{\theta I_{n}, \Omega \Phi_{i} h_{i}\}.$$
Since $\theta$ and $\Phi$ are left coprime, then $\theta$ and $\Phi_{i}$ are left coprime. By (1), we obtain the desired result.
\end{proof}

\section{Subnormality of block Toeplitz operators}
In this section, we study which hyponormal Toeplitz operators are either normal or analytic. This problem is closely related to Halmos's Problem 5.

Suppose $\Phi \in L^{\infty}(\mathbb{M}_{n})$, the authors (Corollary 2.5 in \cite{GHR2006}) show that $\ker H_{\Phi}=\Theta H^{2}(\mathbb{C}^{n})$ for some square inner matrix $\Theta(z)$ if and only if $\Phi=\Phi_{1} \Theta^{\ast}$ where $\Phi_{1} \in H^{\infty}(\mathbb{M}_{n})$, and $\Phi_{1}$ and $\Theta$ are right coprime. Using this result, we can obtain the following lemma.

\begin{lemma}\label{shth4}
Suppose $\Delta$ is a two-sided inner function with values in  $\mathbb{M}_{n}$ without nonconstant scalar inner divisors and $\Phi \in H^{\infty}(\mathbb{M}_{n})$ and $\theta$ is a scalar inner function. If $\theta I_{n}$ and $\Phi$ are left coprime and
\begin{equation}\label{shf49}
  \ker H_{\widetilde{\Delta}\widetilde{\Phi}\widetilde{\Delta}\widetilde{\theta}^{\ast}}=\ker H_{\widetilde{D(\Delta)}\widetilde{\Delta}\widetilde{\theta}^{\ast}},
\end{equation}
Then there exists $\Delta_{1}$ with $D(\Delta) =D(\Delta_{1})$ such that
\begin{equation}\label{shf64}
  \text{left-g.c.d}\{\theta I_{n}, \Delta\Phi \Delta\}=\text{left-g.c.d}\{\theta I_{n}, D(\Delta)\Delta \}=\text{left-g.c.d.}\{\theta I_{n}, \Delta \Delta_{1}\}.
\end{equation}

\end{lemma}
\begin{proof}
By Corollary 2.5 in \cite{GHR2006}, $(\ref{shf49})$ is equivalent to the first equation in $(\ref{shf64})$.
By Proposition \ref{shly16}, there exists $\Delta_{1}$ with $D(\Delta_{1})=D(\Delta)$ such that
$$\text{left-g.c.d}\{\theta I_{n}, \Delta\Phi\Delta\}=\text{left-g.c.d}\{\theta I_{n}, \Delta \Delta_{1}\}.$$
\end{proof}

Let $\Phi \in L^{\infty}(\mathbb{M}_{n})$ be such that $\Phi$ and $\Phi^{\ast}$ are of bounded type.
By Theorem 3.1 in \cite{CHL2021}, we have
\begin{equation}\nonumber
\Phi=\Phi_{1}\Theta^{\ast},~ \Phi^{\ast}= \Phi_{2}\Omega^{\ast}~(\text{right coprime}),
\end{equation}
where $\Theta, \Omega$ are two-sided inner functions with values in  $\mathbb{M}_{n}$, and $\Phi_{1}, \Phi_{2} \in H^{\infty}(\mathbb{M}_{n})$.
If $T_{\Phi}$ is hyponormal, then $\ker H_{\Phi^{\ast}} \subset \ker H_{\Phi}$, i.e., there exists a two-sided inner function $\Delta$ such that
 $\Omega=\Theta \Delta$.

To prove our main theorem, we need the following technical lemma.
\begin{lemma}\label{shly18}
Let $\Phi \in L^{\infty}(\mathbb{M}_{n})$ be such that $\Phi$ and $\Phi^{\ast}$ are of bounded type of the form
\begin{equation}\label{shf27}
  \Phi=\Phi_{1} \theta^{\ast}, ~\Phi^{\ast}=\Phi_{2}\theta^{\ast} \Delta^{\ast}(\text{right coprime}),
\end{equation}
where $\theta$ is a scalar inner function, and $\theta$ is not an inner factor of $D(\Delta)$ and $\Phi_{1},\Phi_{2} \in H^{\infty}(\mathbb{M}_{n})$. If $T_{\Phi}$ and $T^{2}_{\Phi}$ are hyponormal, then $T_{\Phi}$ is either normal or analytic.
 \end{lemma}

\begin{proof}

If $\theta$ is constant, then $\Phi$ is analytic. Hence in the rest of the proof, we always assume $\theta$ is nonconstant and we just need to show that $T_{\Phi}$ is normal.
 Let
\begin{equation}\label{shf29}
\delta=D(\Delta)=\Delta \Delta',
\end{equation}
where $\Delta'$ is a two-sided inner function.

The following arguments are divided into two steps.

Step 1:

It is not hard to check that $\Phi^{2}\theta^{2}\delta^{2} $ and $\Phi^{\ast 2}\theta^{2}\delta^{2} $ are holomorphic. Since $T_{\Phi}$ is hyponormal, by Theorem 3.3 in \cite{GHR2006}, $\Phi $ is normal, i.e., $\Phi(z) \Phi^{\ast}(z)=\Phi(z)^{\ast} \Phi(z)$ for almost every $z \in \mathbb{T}$. Therefore,
\begin{equation}\label{shf25}
  \begin{split}
    T^{\ast}_{\theta^{2} \delta^{2}}[T_{\Phi}^{\ast 2}, T_{\Phi}^{2}]T_{\theta^{2} \delta^{2}}
    &=T^{\ast}_{\theta^{2} \delta^{2}}(T_{\Phi}^{\ast 2}T_{\Phi}^{2}-T_{\Phi}^{2}T_{\Phi}^{\ast 2})T_{\theta^{2} \delta^{2}}\\
    &=T_{\Phi^{\ast 2}\theta^{ \ast 2} \delta^{\ast 2} }T_{\Phi^{2}\theta^{2} \delta^{2}}-T_{\Phi^{ 2}\theta^{ \ast 2} \delta^{\ast 2} }T_{\Phi^{\ast 2}\theta^{2}
    \delta^{2}}\\
    &=T_{\Phi^{\ast 2} \Phi^{2}} -T_{\Phi^{2} \Phi^{\ast 2}}=0.
  \end{split}
\end{equation}
Since $T^{2}_{\Phi}$ is hyponormal, then $(\ref{shf25})$ implies that $[T_{\Phi}^{\ast 2}, T_{\Phi}^{2}]T_{\theta^{2} \delta^{2}}=0$. By $(\ref{eq2})$, we have
\begin{equation}\label{shf26}
  \begin{split}
    0=[T_{\Phi}^{\ast 2}, T_{\Phi}^{2}]T_{\theta^{2} \delta^{2}}
    &=T_{\Phi }^{\ast 2}T_{\Phi^{2}\theta^{2} \delta^{2}}-T_{\Phi }^{ 2}T_{\Phi^{\ast 2}\theta^{2}
    \delta^{2}}\\
    &= T_{\Phi^{\ast } }T_{ \Phi^{\ast }\Phi^{2}\theta^{2} \delta^{2}}-T_{\Phi }T_{ \Phi \Phi^{\ast 2}\theta^{2}
    \delta^{2}}\\
    &=H^{\ast}_{\Phi^{\ast}}H_{\Phi \Phi^{\ast 2} \theta^{2}\delta^{2}}-H^{\ast}_{\Phi} H_{\Phi^{\ast}\Phi^{2} \theta^{2}\delta^{2}}.
  \end{split}
\end{equation}
By $(\ref{shf27})$ and $(\ref{shf29})$,
\begin{equation}\label{shf31}
\Phi \Phi^{\ast 2} \theta^{2} \delta^{2}=\Phi_{1}\Phi_{2} \Delta'\Phi_{2} \Delta'\theta^{\ast},~ \Phi^{2}\Phi^{\ast} \theta^{2}\delta^{2}=\Phi_{1}^{2}\Phi_{2}\Delta'\delta \theta^{\ast},
\end{equation}
which implies that
\begin{equation}\label{shf30}
  \widetilde{\theta}H^{2}(\mathbb{C}^{n}) \subset \ker H^{\ast}_{\Phi \Phi^{\ast 2} \theta^{2}\delta^{2}}~\text{and}~ \widetilde{\theta}H^{2}(\mathbb{C}^{n}) \subset \ker H^{\ast}_{\Phi^{\ast} \Phi^{2} \theta^{2}\delta^{2}}.
\end{equation}
Since
\begin{equation}\label{shf35}
  \ker H^{\ast}_{\Phi^{\ast}}=\ker H^{\ast}_{\Phi_{2} \theta^{\ast} \Delta^{\ast}} =\ker H_{\widetilde{\theta^{\ast} \Delta^{\ast}} \widetilde{\Phi_{2}}} \subset \ker H_{\widetilde{\theta^{\ast} } \widetilde{\Phi_{2}}} = \widetilde{\theta}H^{2}(\mathbb{C}^{n})
\end{equation}
and $\ker H_{\Phi}^{\ast}=\widetilde{\theta}H^{2}(\mathbb{C}^{n})$, then $(\ref{shf30})$ shows that
$$\text{cl-ran}H_{\Phi \Phi^{\ast 2} \theta^{2}\delta^{2}} \subset \mathcal{K}_{\widetilde{\theta}}\perp \ker H^{\ast}_{\Phi^{\ast}}~\text{and}~ \text{cl-ran}H_{\Phi^{2} \Phi^{\ast } \theta^{2}\delta^{2}} \subset \mathcal{K}_{\widetilde{\theta}} \perp \ker H_{\Phi}^{\ast},$$
where $\mathcal{K}_{\widetilde{\theta}}=H^{2}(\mathbb{C}^{n})\ominus \widetilde{\theta}H^{2}(\mathbb{C}^{n}) $.
Therefore,
\begin{equation}\label{shf42}
  \begin{split}
    \ker H_{\Phi \Phi^{\ast 2} \theta^{2}\delta^{2}}
    &=\ker H^{\ast}_{\Phi^{\ast}}H_{\Phi \Phi^{\ast 2} \theta^{2}\delta^{2}}\\
    &=\ker H^{\ast}_{\Phi} H_{\Phi^{\ast} \Phi^{2} \theta^{2}\delta^{2}}\\
    &=\ker H_{\Phi^{\ast} \Phi^{2} \theta^{2}\delta^{2}}.
  \end{split}
\end{equation}
By $(\ref{shf31})$ and $(\ref{shf42})$, we have
\begin{equation}\label{shf32}
  \ker H_{\Phi_{1}\Phi_{2} \Delta'\Phi_{2} \Delta'\theta^{\ast}}=\ker H_{\Phi_{1}^{2}\Phi_{2}\Delta'\delta \theta^{\ast}}.
\end{equation}

Step 2:

Suppose $\Delta$ has no nonconstant scalar inner factors, by Lemma \ref{shly7}, $\Delta'$ has no nonconstant scalar inner factors, and by Lemma \ref{shly9}, $D(\Delta')=\delta$.
Let $$\omega=\text{g.c.d.}\{\theta, \delta\},$$
 by $(\ref{shf32})$, Lemma \ref{shly6} and Lemma \ref{shth4}, (if $\Delta$ has a scalar inner divisor $\delta_{0}$, let $\Delta=\delta_{0} \Delta_{1}$ and $D(\Delta_{1})=\delta_{1}=\Delta_{1}\Delta'$, then $D(\Delta')=\delta_{1}$ and $\Phi_{1}^{2}\Phi_{2}\Delta'\delta \theta^{\ast}=\Phi_{1}^{2}\Phi_{2}\Delta'\delta_{0}\delta_{1} \theta^{\ast}$, combining with Corollary  \ref{shly17}, $(\ref{shf33})$ also holds.)
\begin{equation}\label{shf33}
  \text{left-g.c.d.}\{\widetilde{\theta} I_{n}, \widetilde{\Delta'}\widetilde{\Phi_{2}}\widetilde{\Delta'}\widetilde{\Phi_{1}\Phi_{2} }\}=\text{left-g.c.d.}\{\widetilde{\theta} I_{n},\widetilde{\delta\Delta'}\widetilde{\Phi_{1}^{2}\Phi_{2} }\}=\widetilde{\omega}.
\end{equation}

Since $\theta$ is not the factor of $D(\Delta)$, then $\omega \neq \theta$. Hence $H_{\Phi_{1}^{2}\Phi_{2}\Delta'\delta \theta^{\ast}} \neq 0$.
By $(\ref{shf31})$ and $(\ref{shf33})$, $\widetilde{\omega}^{\ast} \widetilde{\theta}$ and $\widetilde{\omega}^{\ast} \widetilde{\Phi }\widetilde{\Phi^{\ast 2}}\widetilde{\theta^{3}}\widetilde{\delta^{2}}$ are left coprime, by Theorem 4.17 in \cite{CHL2019}, $\widetilde{\omega}^{\ast} \widetilde{\theta}$ and $\widetilde{\omega}^{\ast} \widetilde{\Phi }\widetilde{\Phi^{\ast 2}}\widetilde{\theta^{3}}\widetilde{\delta^{2}}$ are right coprime. By $(\ref{eq1})$, $(\ref{shf31})$ and Corollary 2.5 in \cite{GHR2006},
$$\ker H^{\ast}_{\Phi^{\ast 2} \Phi \theta^{2}\delta^{2}} =\widetilde{\omega}^{\ast} \widetilde{\theta} H^{2}(\mathbb{C}^{n}).$$
Therefore,

\begin{equation}\label{shf36}
\text{cl-ran}H_{\Phi^{\ast 2} \Phi \theta^{2}\delta^{2}}=(\ker H^{\ast}_{\Phi^{\ast 2} \Phi \theta^{2}\delta^{2}})^{\perp}=\mathcal{K}_{\widetilde{\omega}^{\ast} \widetilde{\theta}}.
\end{equation}
Since $T_{\Phi}$ is hyponormal, by Theorem 3.3 in \cite{GHR2006}, there exists $K \in H^{\infty}(\mathbb{M}_{n})$ with $\|K\|_{\infty} \leq 1$ such that
\begin{equation}\label{shf34}
  \Phi -K \Phi^{\ast} \in H^{\infty}(\mathbb{M}_{n}).
\end{equation}
By $(\ref{eq4})$ and $(\ref{shf34})$,
\begin{equation}\label{shf40}
  H_{\Phi}=H_{K\Phi^{\ast}}=T^{\ast}_{\widetilde{K}} H_{\Phi^{\ast}}.
\end{equation}
Since $\Phi \Phi^{\ast}\theta^{2} \delta^{2} \in H^{\infty}(\mathbb{M}_{n})$ and $\Phi$ is normal, then
\begin{equation}\nonumber
  \begin{split}
   H_{\Phi^{\ast} \Phi^{2} \theta^{2}\delta^{2}}
   &=H_{\Phi} T_{\Phi^{\ast} \Phi \theta^{2}\delta^{2}}\\
    &=T^{\ast}_{\widetilde{K}} H_{\Phi^{\ast}}T_{\Phi^{\ast} \Phi \theta^{2}\delta^{2}}\\
    &=T^{\ast}_{\widetilde{K}} H_{\Phi^{\ast 2} \Phi \theta^{2}\delta^{2}}.
  \end{split}
\end{equation}
By $(\ref{shf26})$,
\begin{equation}\nonumber
  \begin{split}
    0
    &=H^{\ast}_{\Phi^{\ast}}H_{\Phi \Phi^{\ast 2} \theta^{2}\delta^{2}}-H^{\ast}_{\Phi} H_{\Phi^{\ast} \Phi^{2} \theta^{2}\delta^{2}}\\
    &=H^{\ast}_{\Phi^{\ast}}H_{\Phi \Phi^{\ast 2} \theta^{2}\delta^{2}}-H^{\ast}_{\Phi^{\ast}}T_{\widetilde{K}} H_{\Phi^{\ast} \Phi^{2} \theta^{2}\delta^{2}}\\
    &=H^{\ast}_{\Phi^{\ast}}(I-T_{\widetilde{K}}T_{\widetilde{K}}^{\ast}) H_{\Phi \Phi^{\ast 2} \theta^{2}\delta^{2}},
  \end{split}
\end{equation}
by $(\ref{shf35})$ and $(\ref{shf36})$, we have
\begin{equation}\label{shf37}
  (I-T_{\widetilde{K}}T_{\widetilde{K}}^{\ast})\mathcal{K}_{\widetilde{\omega}^{\ast} \widetilde{\theta}} \subset \widetilde{\theta}H^{2}(\mathbb{C}^{n}).
\end{equation}
For $f \in \mathcal{K}_{\widetilde{\omega}^{\ast} \widetilde{\theta}}$, $(\ref{shf37})$ shows that there exists $h\in H^{2}(\mathbb{C}^{n})$ such that
$$(I-T_{\widetilde{K}}T_{\widetilde{K}}^{\ast})f= \widetilde{\theta}h \Rightarrow T_{\widetilde{K}}T_{\widetilde{K}}^{\ast}f=f-\widetilde{\theta}h.$$
Since $\mathcal{K}_{\widetilde{\omega}^{\ast} \widetilde{\theta}} \subset \mathcal{K}_{ \widetilde{\theta}}$, then
$$\|T_{\widetilde{K}}T_{\widetilde{K}}^{\ast}f\|^{2}=\|f-\widetilde{\theta}h\|^{2}=\|f\|^{2}+\|\widetilde{\theta}h\|^{2}.$$
Since $T_{\widetilde{K}}$ is contractive, then $\widetilde{\theta}h=0$. Therefore,
\begin{equation}\label{shf38}
  T_{\widetilde{K}}T_{\widetilde{K}}^{\ast}f=f, ~f\in \mathcal{K}_{\widetilde{\omega}^{\ast} \widetilde{\theta}} .
\end{equation}
It is clear that
$$\|f\|=\|T_{\widetilde{K}}T_{\widetilde{K}}^{\ast}f\|\leq \|T_{\widetilde{K}}^{\ast}f\| \leq \|f\|.$$
Therefore, $\|f\|=\|T_{\widetilde{K}}^{\ast}f\| =\|T_{\check{K}}f\|$. Obviously,
$$\|f\|^{2} \geq \|\check{K}f \|^{2}= \|P(\check{K}f)\|^{2}+\|(I-P)(\check{K}f)\|^{2}\geq \|f\|^{2},$$
which yields that $\|(I-P)(\check{K}f)\|^{2}=0$. Hence $\check{K}f \in H^{2}(\mathbb{C}^{n})$. Combining $(\ref{shf38})$, we have
\begin{equation}\label{shf39}
  \widetilde{K}\check{K}f=f, ~f\in \mathcal{K}_{\widetilde{\omega}^{\ast} \widetilde{\theta}} .
\end{equation}
Since $\widetilde{\omega}^{\ast} \widetilde{\theta}$ is not constant, let $u=\widetilde{\omega}^{\ast} \widetilde{\theta}$ and $K_{w}(z)=\frac{1-u(z)u(w)^{\ast}}{1-z\overline{w}}I_{n}$, then there exist $z_{0}, w_{0} \in \mathbb{D}$ and a constant $\alpha\neq 0$ such that $K_{w_{0}}(z_{0})=\alpha I_{n}$. Therefore,
$$\bigvee\{K_{w_{0}}(z_{0})c :  c\in  \mathbb{C}^{n}\}=\mathbb{C}^{n},$$
where $\bigvee$ denotes the closed linear span in $\mathbb{C}^{n}$.
 By $(\ref{shf39})$, we have
 $$(\widetilde{K}\check{K})(z_{0})K_{w_{0}}(z_{0})c=K_{w_{0}}(z_{0}) (\widetilde{K}\check{K})(z_{0}) c =K_{w_{0}}(z_{0})c,$$
 which implies that
 $$(\widetilde{K}\check{K})(z_{0})=\check{K}^{\ast}(z_{0})\check{K}(z_{0})=I_{n}.$$
  Hence $\check{K}(z_{0})$ is unitary. Since $K$ is contractive, then the maximum principle implies that $K$ is a unitary constant.
By $(\ref{eq2})$ and $(\ref{shf40})$
$$[T_{\Phi}^{\ast},T_{\Phi}]=H_{\Phi^{\ast}}^{\ast} H_{\Phi^{\ast}}-H_{\Phi}^{\ast}H_{\Phi}=H_{\Phi^{\ast}}^{\ast}(I-T_{\widetilde{K}}T_{\widetilde{K}}^{\ast}) H_{\Phi^{\ast}}=0,$$
then $T_{\Phi}$ is normal.

\end{proof}

\begin{remark}
If $\theta$ is a factor of $\delta$, then $(\ref{shf33})$ shows that $\omega= \theta$. Hence $\mathcal{K}_{\widetilde{\omega}^{\ast} \widetilde{\theta}}=0$. In this case, we don't obtain that $K$ is a unitary constant.
Therefore, we need to change some details in the proof of Lemma \ref{shly18}.
\end{remark}

\begin{remark}\label{shre2}
One of the key points in the proof of Lemma \ref{shly18} is $(\ref{shf33})$. If $\theta$ is a divisor of $\delta$, and $\widetilde{\Delta'}\widetilde{\Phi_{2}}\widetilde{\Delta'}\widetilde{\Phi_{1}\Phi_{2} }$ has a scalar inner divisor $\widetilde{\theta}$, then by $(\ref{shf31})$, both $\Phi \Phi^{\ast 2} \theta^{2} \delta^{2}$ and $\Phi^{2} \Phi^{\ast } \theta^{2} \delta^{2}$ are analytic. Therefore, the argument of $(\ref{shf37})$ in Step 2 will fail. Hence the condition ``$\theta$ is not a factor of $D(\Delta)$" seems necessary in the above proof or one can show that $\widetilde{\Delta'}\widetilde{\Phi_{2}}\widetilde{\Delta'}\widetilde{\Phi_{1}\Phi_{2} }$  has no scalar inner factor. However, for example, let
\begin{equation}\nonumber
\Theta=\frac{\sqrt{2}}{2}\left(
  \begin{array}{cc}
    z^{3} & -z^{5}\\
    1 & z^{2} \\
  \end{array}
\right),
~\Delta=\left(
\begin{array}{cc}
    \delta_{1} & \\
     & \delta_{2} \\
  \end{array}
\right),
\end{equation}
where $\delta_{1}$ and $\delta_{2}$ are coprime scalar inner functions. It is clear that $\Theta$ has no scalar inner divisor and $\Theta\Delta \neq \Delta \Theta$. However,
\begin{equation}\nonumber
\Theta \Delta \Theta=\frac{1}{2}z^{2}\left(
  \begin{array}{cc}
    \delta_{1}z^{4}-\delta_{2}z^{3} & -\delta_{1}z^{6}-\delta_{2}z^{5} \\
    \delta_{1}z+\delta_{2} &  -\delta_{1}z^{3}+\delta_{2}z^{2}\\
  \end{array}
\right),
\end{equation}
and it has a scalar inner divisor.
\end{remark}

%

Our main theorem of this section now follows:

\begin{theorem}\label{shth2}
Let $\Phi \in L^{\infty}(\mathbb{M}_{n})$ be such that $\Phi$ and $\Phi^{\ast}$ are of bounded type of the form
\begin{equation}\nonumber
  \Phi=\Phi_{1} \theta^{\ast},~\Phi^{\ast}=\Phi_{2}\theta^{\ast} \Delta^{\ast} ~(\text{right coprime}),
\end{equation}
where $\theta$ is a nonconstant scalar inner function and $\Phi_{1},\Phi_{2} \in H^{\infty}(\mathbb{M}_{n})$. If $T_{\Phi}$ and $T^{2}_{\Phi}$ are hyponormal, then $T_{\Phi}$ is either normal or analytic.
 \end{theorem}
 \begin{proof}
 Assume $\theta^{k}$ is a factor of $D(\Delta)$, then $k$ is finite.  The following arguments are divided into two cases: Case I, $D(\Delta)\neq \theta^{k}$ and $\theta^{k+1}$ is not the factor of $D(\Delta)$; Case II, $D(\Delta)=\theta^{k}$.

 If $k=0$, the result is Lemma \ref{shly18}.

 If $k\geq 1$, then $\Phi_{1}^{2}\Phi_{2}\Delta' \delta \theta^{\ast}$ is analytic. By $(\ref{shf32})$ in the proof of Lemma \ref{shly18}, $\Phi_{1}\Phi_{2} \Delta'\Phi_{2} \Delta'\theta^{\ast}$ is analytic. Since $\Phi_{1}$ and $\theta$ are right coprime, by Theorem 4.17 in \cite{CHL2019}, $\Phi_{1}$ and $\theta$ are left coprime, by Corollary \ref{shly28}, $\Phi_{2} \Delta'\Phi_{2} \Delta'\theta^{\ast}$ is analytic.
Hence
$$\Phi^{2} \theta \delta^{2}=\Phi_{1}^{2}(\theta^{\ast}\delta^{2})~\text{and}~\Phi^{\ast 2} \theta \delta^{2}=\Phi_{2}\Delta'\Phi_{2}\Delta' \theta^{\ast}$$
are holomorphic. Therefore, we replace $\theta^{2} \delta^{2}$ with $\theta \delta^{2}$, and repeat Step 1 in the proof of Lemma \ref{shly18}. The same calculation shows that $(\ref{shf25})$ and $(\ref{shf26})$ hold  for $\theta \delta^{2}$.  By further calculation, 
\begin{equation}\label{shf41}
\Phi \Phi^{\ast 2} \theta \delta^{2}=\Phi_{1}\Phi_{2} \Delta'\Phi_{2} \Delta'\theta^{\ast 2},~ \Phi^{2}\Phi^{\ast} \theta\delta^{2}=\Phi_{1}^{2}\Phi_{2}\Delta'\delta \theta^{\ast 2}.
\end{equation}
Since $\Phi_{1}\Phi_{2} \Delta'\Phi_{2} \Delta'\theta^{\ast}$ is analytic, then $(\ref{shf30})$ is changed as follows:
\begin{equation}\nonumber
  \widetilde{\theta}H^{2}(\mathbb{C}^{n}) \subset \ker H^{\ast}_{\Phi \Phi^{\ast 2} \theta\delta^{2}}~\text{and}~ \widetilde{\theta}H^{2}(\mathbb{C}^{n}) \subset \ker H^{\ast}_{\Phi^{\ast} \Phi^{2} \theta\delta^{2}}.
\end{equation}
Therefore,
$$\text{cl-ran}H_{\Phi \Phi^{\ast 2} \theta\delta^{2}} \subset \mathcal{K}_{\widetilde{\theta}}\perp \ker H^{\ast}_{\Phi^{\ast}}~\text{and}~ \text{cl-ran}H_{\Phi^{2} \Phi^{\ast } \theta\delta^{2}} \subset \mathcal{K}_{\widetilde{\theta}} \perp \ker H_{\Phi}^{\ast},$$
and $(\ref{shf42})$ is changed as follows:
\begin{equation}\label{shf44}
  \begin{split}
    \ker H_{\Phi \Phi^{\ast 2} \theta\delta^{2}}
    &=\ker H_{\Phi^{\ast} \Phi^{2} \theta\delta^{2}}.
  \end{split}
\end{equation}
If $k=1$.

For Case I:
Since $\theta^{ 2}$ is not a factor of $\delta$, then turn to Step 2 in the proof of Lemma \ref{shly18}.

In this case,  by $(\ref{shf41})$, $\Phi^{\ast} \Phi^{2} \theta\delta^{2}$ is not analytic, which means that $H_{\Phi^{\ast} \Phi^{2} \theta\delta^{2}}\neq 0$.  Let $\omega=\text{g.c.d.}\{\theta^{2}, \delta\}$, then $$\omega=\theta \text{g.c.d.}\{\theta^{}, \theta^{\ast}\delta\} \neq \theta^{2} .$$
Therefore,
$(\ref{shf33})$ is changed as follows:
\begin{equation}\nonumber
  \text{left-g.c.d.}\{\widetilde{\theta}^{2} I_{n}, \widetilde{\Delta'}\widetilde{\Phi_{2}}\widetilde{\Delta'}\widetilde{\Phi_{2}\Phi_{1} }\}=\text{left-g.c.d.}\{\widetilde{\theta}^{2} I_{n},\widetilde{\delta\Delta'}\widetilde{\Phi_{1}^{2}\Phi_{2} }\}=\widetilde{\omega},
\end{equation}
and
\begin{equation}\nonumber
\text{cl-ran}H_{\Phi^{\ast 2} \Phi \theta\delta^{2}}=\mathcal{K}_{\widetilde{\omega}^{\ast} \widetilde{\theta^{2}}} \subset \mathcal{K}_{\ \widetilde{\theta}}.
\end{equation}
Since $\omega \neq \theta^{2}$, then $\mathcal{K}_{\widetilde{\omega}^{\ast} \widetilde{\theta^{2}}}\neq \{0\}$.
The rest of the proof is similar to Step 2 in Lemma \ref{shly18}'s proof, and we omit it.

For Case II:  By Lemma \ref{shth4}, there exists $\Delta_{1}$ with $ D(\Delta_{1})=\widetilde{\delta}$ such that
\begin{equation}\nonumber
  \text{left-g.c.d.}\{\widetilde{\theta}^{2} I_{n}, \widetilde{\Delta'}\widetilde{\Phi_{2}}\widetilde{\Delta'}\widetilde{\Phi_{2}\Phi_{1} }\}=\text{left-g.c.d.}\{\widetilde{\theta}^{2} I_{n}, \widetilde{\Delta'}\Delta_{1}\}=\widetilde{\Delta'}\Delta_{1}~\text{(by $\theta=\delta$)}.
\end{equation}
However,
$$\text{left-g.c.d.}\{\widetilde{\theta}^{2} I_{n},\widetilde{\delta\Delta'}\widetilde{\Phi_{1}^{2}\Phi_{2} }\}=\widetilde{\delta\Delta'}~\text{(by $\theta=\delta$)}.$$
By Lemma \ref{shly7} and Proposition \ref{shly14}, $\Delta_{1} \neq \widetilde{\delta}$. Therefore
$$ \text{left-g.c.d.}\{\widetilde{\theta}^{2} I_{n}, \widetilde{\Delta'}\widetilde{\Phi_{2}}\widetilde{\Delta'}\widetilde{\Phi_{2}\Phi_{1} }\} \neq  \text{left-g.c.d.}\{\widetilde{\theta}^{2} I_{n},\widetilde{\delta\Delta'}\widetilde{\Phi_{1}^{2}\Phi_{2} }\}.$$
This contradicts $(\ref{shf44})$. Hence $\delta \neq \theta$.

If $k>1$, $(\ref{shf41})$ and $(\ref{shf44})$ imply that $\Phi_{1}\Phi_{2}\Delta'\Phi_{2} \Delta'\theta^{\ast 2}$ is holomorphic. By Corollary \ref{shly28}, $\Phi_{2}\Delta'\Phi_{2} \Delta'\theta^{\ast 2}$ is analytic. Therefore,
$$\Phi^{\ast 2} \delta^{2}=\Phi_{2}\Delta'\Phi_{2} \Delta'\theta^{\ast 2}~\text{and}~\Phi^{ 2} \delta^{2} =\Phi_{1}^{2}(\theta^{\ast 2} \delta^{2})$$
are holomorphic. We replace $\theta\delta^{2}$ with $\delta^{2}$, and repeat Step 1. If $k=2$, for Case I, we turn to Step 2 and complete the proof. For Case II,
$(\ref{shf44})$ is changed as follows:
\begin{equation}\label{shf53}
      \ker H_{\Phi \Phi^{\ast 2} \delta^{2}}
    =\ker H_{\Phi^{\ast} \Phi^{2} \delta^{2}}.
\end{equation}
Since $$ \Phi \Phi^{\ast 2} \delta^{2}=\Phi_{1}\Phi_{2}\Delta'\Phi_{2} \Delta'\theta^{\ast 3}~\text{and}~ \Phi^{\ast}\Phi^{ 2}\delta^{2} =\Phi_{2}\Phi_{1}^{2} \Delta'(\theta^{\ast 3} \delta),$$
then $(\ref{shf53})$ means that
\begin{equation}\label{shf54}
  \text{left-g.c.d.}\{\widetilde{\theta}^{3} I_{n}, \widetilde{\Delta'}\widetilde{\Phi_{2}}\widetilde{\Delta'}\widetilde{\Phi_{2}\Phi_{1} }\} =  \text{left-g.c.d.}\{\widetilde{\theta}^{3} I_{n},\widetilde{\delta\Delta'}\widetilde{\Phi_{1}^{2}\Phi_{2} }\}.
\end{equation}
By Lemma \ref{shly6} and $\delta=\theta^{2}$,
\begin{equation}\nonumber
  \text{left-g.c.d.}\{\widetilde{\theta}^{3}I_{n},\widetilde{\delta\Delta'}\widetilde{\Phi_{1}^{2}\Phi_{2} }\}=\text{left-g.c.d.}\{\widetilde{\theta}^{3}I_{n},\widetilde{\delta\Delta'}\}=\widetilde{\delta}~\text{left-g.c.d.}\{\widetilde{\theta}I_{n} ,\widetilde{\Delta'}\}.
\end{equation}
Then by Lemma \ref{shth4}, there exists $\Delta_{1}, D(\Delta_{1})=D(\widetilde{\Delta'})$ such that
\begin{equation}\nonumber
  \text{left-g.c.d.}\{\widetilde{\theta}^{3}I_{n}, \widetilde{\Delta'}\widetilde{\Phi_{2}}\widetilde{\Delta'}\widetilde{\Phi_{1}\Phi_{2} }\}=\text{left-g.c.d.}\{\widetilde{\theta}^{3}I_{n}, \widetilde{\Delta'}\Delta_{1}\}.
\end{equation}
The above equations show that
\begin{equation}\label{snf73}
  \text{left-g.c.d.}\{\widetilde{\theta}^{3}I_{n},\widetilde{\delta\Delta'}\}=\text{left-g.c.d.}\{\widetilde{\theta}^{3}I_{n}, \widetilde{\Delta'}\Delta_{1}\}.
\end{equation}
Since $D(\widetilde{\Delta'})=D(\Delta_{1})$, by Lemma \ref{shly7} and Proposition \ref{shly14} both $\Delta_{1}$ and $\Delta'$  have no nonconstant scalar inner factor, then by Theorem \ref{shly10}, if $(\ref{snf73})$ holds, then $\Omega^{\ast}\widetilde{\theta}^{3} $ and $\widetilde{\delta}$ are coprime, where $\Omega=\text{g.c.d.}\{\widetilde{\theta}^{3}, \widetilde{\delta} \}$. Since $\delta=\theta^{2}$, then $\Omega=\widetilde{\theta}^{2}$. This shows that $\Omega^{\ast}\widetilde{\theta}^{3} $ and $\widetilde{\delta}$ are not coprime. This contradicts the fact that $(\ref{snf73})$ holds. Hence this case doesn't happen.

In general, we can show that Case II does not appear for any $k$. If $\theta^{k}=\delta$, then  $(\ref{shf54})$ is changed as follows:
 \begin{equation}\label{shf66}
  \text{left-g.c.d.}\{\widetilde{\theta}^{k+1}I_{n}, \widetilde{\Delta'}\widetilde{\Phi_{2}}\widetilde{\Delta'}\widetilde{\Phi_{2}\Phi_{1} }\} =  \text{left-g.c.d.}\{\widetilde{\theta}^{k+1}I_{n},\widetilde{\delta\Delta'}\widetilde{\Phi_{1}^{2}\Phi_{2} }\}.
\end{equation}
 It is clear that
$$\text{left-g.c.d.}\{\widetilde{\theta}^{k+1}I_{n},\widetilde{\delta\Delta'}\widetilde{\Phi_{1}^{2}\Phi_{2} }\}=\widetilde{\delta}~\text{left-g.c.d.}\{\widetilde{\theta}I_{n},\widetilde{\Delta'} \}.$$
The similar argument for $k=2$, we obtain that
$$\Omega=\text{g.c.d.}\{\widetilde{\theta}^{k+1}, \widetilde{\delta} \}=\widetilde{\delta}=\widetilde{\theta}^{k}.$$
Then $\Omega^{\ast}\widetilde{\theta}^{k+1} $ and $\widetilde{\delta}$ are not coprime. By Theorem \ref{shly10}, this contradicts the fact that $(\ref{shf66})$ holds.

If $k>2$, we replace $\delta^{2}$ with $\theta^{\ast}\delta^{2}$ and repeat Step 1.
We summarize the above argument as follows:

\vspace{0.2cm}

We begin with $\theta^{2}\delta^{2}$ and execute Step 1. From $(\ref{shf31})$, it follows that $\Phi^{2}\Phi^{\ast} \theta^{2}\delta^{2}$ is analytic, but then Step 2 fails. Using $(\ref{shf32})$, we find that
$\Phi^{2}\theta \delta^{2}$ and $\Phi^{\ast 2}\theta \delta^{2}$ are analytic.

For the first iteration, we replace $\theta^{2} \delta^{2}$ with $\theta \delta^{2}$ and repeat Step 1. We  obtain that
$\Phi^{2} \delta^{2}$ and $\Phi^{\ast 2} \delta^{2}$ are analytic.

$~~\vdots$

For the $j$th($j<k$) iteration, letting $\theta^{-1}=\theta^{\ast}$, based on the $j-1$th iteration, we find that $\Phi\theta^{3-j}\delta^{2}$ and $\Phi^{\ast}\theta^{3-j}\delta^{2}$ are analytic.
Replacing $\theta^{3-j}\delta^{2}$ with $\theta^{2-j}\delta^{2}$, we repeat Step 1 and obtain:
$$\Phi^{2}\Phi^{\ast}\theta^{2-j}\delta^{2}=\Phi_{1}^{2}\Phi_{2}\Delta'(\theta^{\ast j+1} \delta)$$
is analytic. By a similar argument, we get that $\Phi^{2} \theta^{2-(j+1)} \delta^{2}$ and $\Phi^{\ast 2} \theta^{2-(j+1)} \delta^{2}$ are analytic.

$~~\vdots$

For the $k$th iteration,  replacing $\theta^{3-k}\delta^{2}$ with $\theta^{2-k}\delta^{2}$, we find
$$\Phi^{2} \Phi^{\ast} \theta^{2-k} \delta^{2}=\Phi_{1}^{2}\Phi_{2}\Delta'\delta \theta^{\ast k+1},$$
which is not analytic.
Since Case II doesn't occur, we turn to Step 2, and we obtain the desired result.
\end{proof}
It is worth mentioning that the above theorem may fail if we drop the assumption ``right coprime". The counter example can be found in \cite{CHL2012} (see Remark 4.7).

\begin{corollary}
Let $\Phi \in L^{\infty}(\mathbb{M}_{n})$ be such that $\Phi$ and $\Phi^{\ast}$ are of bounded type of the form
\begin{equation}\nonumber
  \Phi=\Phi_{1} \theta^{\ast}, ~(\text{right coprime}),
\end{equation}
where $\theta$ is a scalar inner function and $\Phi_{1} \in H^{\infty}(\mathbb{M}_{n})$. Then the $2$-hyponormality and subnormality for $T_{\Phi}$ are equivalent.
 \end{corollary}

\subsection*{Acknowledgements}
The authors would like to thank the referee for his/her careful reading of the paper and helpful suggestions. This first author was supported by National Natural Science Foundation of China
(No. 12001082) and STU Scientific Research Initiation Grant (No. NTF23012), the second author was supported  by the National Natural Science Foundation of China (No. 12031002), the  third  author was supported  by the National Natural Science Foundation of China (No. 12401151) and the National Postdoctoral Researcher Program (Grant No.GZB20240100).

\subsection*{Conflict of Interests}
The authors declare that they have no conflict of interest.


\end{document}